\documentclass[11pt,a4paper]{amsart}

\title[A Dolbeault-Grothendieck lemma on complex spaces]{A Dolbeault-Grothendieck lemma on 
complex spaces via Koppelman formulas}

\author{Mats Andersson \& H{\aa}kan Samuelsson}
\thanks{The first author was partially supported by a grant from the Swedish Research Council. The second author
wishes to thank the Department of Mathematics, University of Oslo, where parts of his work was done.}
\address{M. Andersson, Department of Mathematical Sciences, Division of Mathematics, University of Gothenburg and 
Chalmers University of Technology, SE-412 96 G\"{o}teborg, Sweden}
\email{matsa@chalmers.se}
\address{H. Samuelsson, Department of Mathematical Sciences, Division of Mathematics, University of Gothenburg and 
Chalmers University of Technology, SE-412 96 G\"{o}teborg, Sweden}
\email{haakan.samuelsson@gmail.com}
\usepackage{amsmath}
\usepackage{amsthm,amssymb,latexsym}
\usepackage{url,ae}
\usepackage{t1enc}
\usepackage{fancyheadings}
\usepackage{mathrsfs}
\usepackage{graphicx}
\usepackage{eufrak} 
\usepackage{geometry}
\newtheorem{proposition}{Proposition}[section]
\newtheorem{theorem}[proposition]{Theorem}
\newtheorem{lemma}[proposition]{Lemma}
\newtheorem{corollary}[proposition]{Corollary}

\theoremstyle{definition}
\newtheorem{definition}[proposition]{Definition}
\newtheorem{example}[proposition]{Example}
\newtheorem{remark}[proposition]{Remark}

\newcommand{\Image}{{\text{Im}\,}}
\newcommand{\Ker}{{\text{Ker}\,}}
\newcommand{\C}{\mathbb{C}}
\newcommand{\debar}{\bar{\partial}}
\newcommand{\dbar}{\bar{\partial}}
\newcommand{\A}{\mathscr{A}}
\newcommand{\R}{\mathbb{R}}
\newcommand{\J}{\mathcal{J}}
\newcommand{\E}{\mathscr{E}}
\newcommand{\Curr}{\mathcal{C}}
\newcommand{\W}{\mathcal{W}}
\newcommand{\PM}{\mathcal{PM}}

\newcommand{\hol}{\mathscr{O}}
\newcommand{\Ok}{\mathscr{O}}
\newcommand{\ordo}{\mathcal{O}}

\newcommand{\K}{\mathscr{K}}
\newcommand{\D}{\mathscr{D}}
\newcommand{\F}{\mathscr{F}}
\newcommand{\Proj}{\mathscr{P}}
\newcommand{\real}{\mathfrak{R}\mathfrak{e}}
\newcommand{\CH}{\mathcal{{ CH}}}
\newcommand{\B}{\mathbb{B}}
\newcommand{\pmm}{pseudomeromorphic }
\newcommand{\nbh}{neighborhood }
\newcommand{\1}{{\bf 1}}
\newcommand{\w}{{\wedge}}
\newcommand{\codim}{{\text{codim}\,}}
\newcommand{\Ba}{{\mathcal B}}
\newcommand{\Homs}{{\mathcal Hom}}

\newcommand{\Ho}{{\mathcal H}}
\newcommand{\Ext}{{\mathcal Ext}}
\newcommand{\Kers}{{\mathcal Ker\,}}
\newcommand{\hra}{\hookrightarrow}

\def\newop#1{\expandafter\def\csname #1\endcsname{\mathop{\rm #1}\nolimits}}
\newop{span}

\numberwithin{equation}{section}

\begin{document}
\nocite{*}
\bibliographystyle{plain}

\begin{abstract}
Let $X$ be a complex space of pure dimension. We introduce fine sheaves $\A^X_q$ of $(0,q)$-currents,
which coincides with the sheaves of smooth forms on the regular part of $X$, so that the associated 
Dolbeault complex yields a resolution
of the structure sheaf $\hol^X$. Our construction is based on intrinsic and quite explicit 
semi-global Koppelman formulas.    
\end{abstract}

\maketitle
\thispagestyle{empty}

\section{Introduction}
The fundamental Dolbeault-Grothendieck lemma assures that a smooth $\debar$-closed $(0,q)$-form on a complex 
$n$-dimensional manifold
$X$ is locally $\debar$-exact if $q\geq 1$ and holomorphic if $q=0$. This means precisely that the sheaf complex 
\begin{equation}\label{glattupplosn}
0\to \hol^X \hookrightarrow \E^X_{0,0} \stackrel{\debar}{\longrightarrow}
\E^X_{0,1} \stackrel{\debar}{\longrightarrow} \cdots \stackrel{\debar}{\longrightarrow}
\E^X_{0,n} \to 0
\end{equation}
is exact, where $\E^X_{p,*}$ denote the fine sheaves of smooth $(p,*)$-forms on $X$; thus \eqref{glattupplosn}
-- the Dolbeault complex -- is a fine resolution of the structure sheaf $\hol^X$. That a sheaf is fine means 
that it is a 
module over $\E^X_{0,0}$ and thus partitions of unity are possible. If $E\to X$ is a holomorphic 
vector bundle and $\hol(E)$ is the associated locally free $\hol^X$-module we then get, by the abstract 
de~Rham theorem,
the generalized Dolbeault-Grothendieck isomorphism, i.e., 
the representation of the sheaf cohomology groups $H^q(X,\hol(E))$ as the obstruction
to global solvability of the $\debar$-equation.

The main result of this paper is a generalization of the Dolbeault-Grothendieck lemma to reduced complex spaces. 
To achieve this we construct intrinsic quite explicit semi-global weighted Koppelman formulas on 
subvarieties $X\subset\Omega$ of pseudoconvex domains $\Omega\subset \C^N$.

Abstractly, one can always find fine resolutions of $\hol^X$ as was realized by R.\ Godement and A.\ Grothendieck
in the '50s and '60s, but these resolutions are naturally very abstract.
More concrete Dolbeault resolutions of $\hol^X$ 
were constructed by Ancona-Gaveau for certain complex spaces in \cite{AG1}, \cite{AG2}, \cite{AG3};
in particular, their sheaves coincide with $\E_{0,*}^X$ on the regular part of $X$. 
Ancona and Gaveau assume that $X$ is normal, that $X_{sing}$ is a smooth manifold, and 
that there exists a desingularization of $X$ with certain hypotheses on the exceptional divisor. 
Their sheaves of differential forms are then the differential forms on this desingularization specified by jet
conditions along the exceptional divisor.

\begin{center}
---
\end{center}

\noindent Let $X$ be a complex space of pure dimension $n$ and let $\hol^X$ be the structure sheaf of 
(strongly) holomorphic functions. Locally, $X$ can be considered as a subvariety of a domain $\Omega\subset \C^N$,
$i\colon X\hookrightarrow \Omega$, and then $\hol^X=\hol^{\Omega}/\J_X$, where $\J_X$ is the sheaf in $\Omega$ 
of holomorphic functions vanishing on $X$. In the same way we say that $\phi$ is a smooth $(p,q)$-form on $X$,
$\phi\in \E_{p,q}(X)$, if given a local embedding, there is a smooth form $\Phi$ in a neighborhood in the ambient 
space such that $\phi=i^*\Phi$ on $X_{reg}$. It is well-known that this defines an intrinsic sheaf $\E^X_{p,q}$ on $X$.

It was proved by Malgrange, \cite{Mal}, that if a smooth function $\phi\in \E_{0,0}(X)$ is $\debar$-closed, i.e., 
holomorphic, on $X_{reg}$, then $\phi$ is indeed (strongly) holomorphic. If $\phi$ is a smooth $\debar$-closed
$(0,q)$-form, $q\geq 1$, on $X$ and $X$ is embedded as a reduced complete intersection in a pseudoconvex domain,
then Henkin-Polyakov, \cite{HePo}, proved that there is a solution $u$ to $\debar u=\phi$ on $X_{reg}$
(which is not Stein in general).
It was an open question for long whether this holds in more generality. In \cite{ASarXiv} we proved\footnote{This 
material is published in \cite{AS}.}  that if
$\phi$ is a smooth $\debar$-closed $(0,q)$-form on a Stein space then there is a solution to the 
$\debar u=\phi$ on $X_{reg}$. 
The proof is based on Koppelman formulas on $X$, and 
the smooth solution to $\dbar u=\phi$ in $X_{reg}$ is given (semiglobally) 
by an intrinsic integral formula $u(z)=\K\phi(z)$, cf., Theorem~\ref{koppelman} and  \eqref{KochP} below.
Similar, but non-intrinsic formulas were obtained by Henkin-Polyakov, \cite{HePo}, 
in the special
case of a reduced, embedded complete intersection. In \cite{AG4}, Koppelman formulas of 
Bochner-Martinelli type are used
but these do not produce solutions to the $\debar$-equation. In the recent paper \cite{HePo2}
the $\dbar$-equation is studied with integral formulas on a not necessarily reduced
projective variety that is a (locally) complete  intersection.

\begin{example}
Let $X$ be the germ of the curve at $0\in \C^2$ defined by $t\mapsto (t^3,t^7+t^8)$. 
If $\phi=\bar{w}d\bar{z}=3(\bar{t}^9+\bar{t}^{10})d\bar{t}$ than one can verify that there is no solution
$\psi=f(t^3,t^7+t^8)$, with $f$ smooth, to $\debar \psi = \phi$. See, e.g., \cite{Rupp} for other examples.
\end{example}

From the example it is clear that our solution $u=\K\phi$ cannot have a smooth extension to $X$.
However, it turns out that it has a current extension to all of $X$. 
For the definition of currents on $X$, see Section~\ref{PMsektion}. It turns out that the singularities
of  $u$ at $X_{sing}$ are quite ``mild''; in particular we can multiply it with any
smooth form and plug it into  the integral operator $\K$ again so that the Koppelman formula
still holds.  The resulting currents can again be multiplied by a smooth form and plugged into
the formulas, etc.  The currents obtained in this way in a finite number of steps 
form  sheaves $\A_q^X$; for the precise definition, see Section~\ref{asec}.
Here is our main result.

\begin{theorem}\label{main}
Let $X$ be a reduced complex space of pure dimension $n$. The sheaves 
$\A^X_{q}$ are fine sheaves of $(0,q)$-currents on $X$, they contain $\E^X_{0,q}$, and moreover
\begin{itemize}
\item[(i)] $\A^X:=\bigoplus_q \A^X_q$ is a module over $\E^X_{0,*}$,

\item[(ii)] $\A^X_q|_{X_{reg}}=\E^X_{0,q}|_{X_{reg}}$,

\item[(iii)] the complex 
\begin{equation*}
0\to \hol^X \hookrightarrow \A^X_{0} \stackrel{\debar}{\longrightarrow}
\A^X_{1} \stackrel{\debar}{\longrightarrow} \cdots \stackrel{\debar}{\longrightarrow}
\A^X_{n} \to 0
\end{equation*}
is exact.
\end{itemize}
\end{theorem}

By the abstract theorem of de~Rham we immediately get the following 
\begin{corollary}\label{kohomrepr}
Let $X$ be a reduced pure-dimensional complex space, let $E\to X$ be a holomorphic vector bundle, and let 
$\hol(E)$ be the associated locally free $\hol^X$-module. Then we have the representation of sheaf cohomology
\begin{equation*}
H^q(X,\hol(E)) \simeq \frac{\{\varphi\in \A_{q}(X,E);\, \, \debar \varphi=0\}}{\{\debar u;\,\, u\in \A_{q-1}(X,E)\}},
\quad q\geq 1.
\end{equation*} 
\end{corollary}

%%By definition, $\A_k$ are the smallest sheaves that are closed under multiplication by smooth forms
%%and the operators  $\K$. 
%%If $\mbox{dim}\, X=1$ and \eqref{glattupplosn} is exact,
%%then $\A^X=\E^X_{0,*}$; see Section~\ref{kurvsektion}. 

The semiglobal Koppelman formulas from \cite{ASarXiv} hold for currents
in  $\A^X$.

\begin{theorem}\label{koppelman}
Let $X$ be an analytic subvariety of pure dimension of a pseudoconvex domain $\Omega \subset \C^N$; take any 
relatively compact subdomain $\Omega'\Subset \Omega$ and put $X'=X\cap \Omega'$. There are integral operators
$\K\colon \A_{q+1}(X)\to \A_q(X')$ and $\Proj\colon \A_0(X) \to \hol(\Omega')$ such that
\begin{equation}\label{Koppel1}
\phi(z)=\debar \K \phi\, (z) + \K(\debar\phi) (z), \quad z\in X', \,\,\phi\in \A_q(X), \,\,q\geq 1,
\end{equation}
\begin{equation}\label{Koppel2}
\phi(z) = \K(\debar \phi) (z) + \Proj\phi\, (z), \quad z\in X', \,\, \phi\in \A_0(X).
\end{equation}
\end{theorem}

The operators $\K$ and $\Proj$ are given as
\begin{equation}\label{KochP}
\K\phi\, (z) = \int_{\zeta} k(\zeta,z)\wedge \phi(\zeta),\quad
\Proj \phi \, (z) = \int_{\zeta} p(\zeta,z)\wedge \phi(\zeta),
\end{equation}
where $k$ and $p$ are integral kernels on $X\times X'$ and $X\times \Omega'$ respectively. These kernels are 
locally integrable with respect to $\zeta$ on $X_{reg}$ and the integrals are principal values at $X_{sing}$. If
$\phi$ vanishes in a neighborhood of an arbitrary point $x\in X$, then $\K\phi$ is smooth at $x$.

\begin{center}
---
\end{center}

\noindent 
Any meromorphic $(k,0)$-form $\phi$ on $X$, defines a principal value current on $X$ 
 and $\debar \phi=0$ means  that $\phi$ is in  $\Ba^X_k$, i.e.,  holomorphic in the 
sense of Barlet-Henkin-Passare, cf., Example~\ref{ba} below.
The sheaf $\Ba^X_k$ is in general larger than the sheaf of (strongly) holomorphic $(k,0)$-forms.
Let  $\W_{0,q}^X$ be the sheaf of pseudomeromorphic $(0,q)$-currents $\mu$  with the so-called SEP,
the {\it standard extension property}. The SEP means roughly speaking that $\mu$ is determined by its values
on any dense Zariski open subset of  $X$. For instance, 
$\W^X$ contains all semi-meromorphic forms as well as push-forwards of such forms under modifications and 
simple projections $X\times Y\to X$.

For currents in $\W^X_{0,*}$ we introduce 
a stronger notion for the $\debar$-operator, denoted $\debar_X$, such that the following holds:
If $\phi\in\W_{0,0}(X)$ and $\debar_X\phi=0$ then   $\phi\in \hol(X)$. 
If $\phi \in \W_{0,q}(X)$, $q\geq 1$, and $X$ is Stein 
then there is a solution to the $\debar$-equation on $X_{reg}$; 
cf.\ Proposition~\ref{Wprop} below and the succeeding remarks.
To explain the operator $\debar_X$, let us for simplicity here assume that $X$ is Cohen-Macaulay
and embedded in a pseudoconvex domain $\Omega \subset \C^N$. From a locally free 
resolution of $\hol^X=\hol^{\Omega}/\J_X$ we derive,
following the ideas in \cite{AW1}, a vector-valued meromorphic $(n,0)$-form $\omega$ on $X$ such that 
$\dbar\omega=0$, i.e, $\omega$ is in $\Ba^X_n$. 
%%
%%%%
We will show that the product $u\wedge \omega$ has a reasonable meaning
when $u$ is in $\W^X$. We say that a current $u$ in $\W^X_{0,*}$ is in $\mbox{Dom}\, \debar_X$ 
if there is  $\phi\in \W^X_{0,*+1}$ such that
$\debar (u\wedge \omega) = \phi\wedge \omega$; we write this equation as 
$\debar_X u = \phi$. 
It  turns out that then $\dbar_X u=\phi=\dbar u$, that  also $\phi$   is in $\mbox{Dom}\, \debar_X$ and 
that $\debar_X \phi=0$.
Since $\dbar\omega=0$ it follows directly that any smooth $(0,*)$-form $u$ is in $\mbox{Dom}\, \debar_X$
(and that $\dbar_X u=\dbar u$). 
In general,
the condition for $u$ being in $\mbox{Dom}\, \debar_X$ can be interpreted as a 
subtle boundary condition on $u$ at $X_{sing}$,
see Lemma \ref{urdjur}. In case $X$ is a reduced complete intersection, our definition
coincides with the definition of $\dbar$ in \cite{HePo2}. 

\begin{proposition}\label{Wprop} We use the notation from Theorem~\ref{koppelman}.
Let $\phi$ be in $\W_{0,q}(X)$. If $q\geq 1$, then $\K \phi$ is defined and is in $\W_{0,q-1}(X')$; 
if $q=0$, then $\Proj \phi$ is defined and it is in $\hol(\Omega')$. 
If $\phi$ is in $\mbox{Dom}\,\, \debar_X$,
then the Koppelman formulas \eqref{Koppel1} and \eqref{Koppel2} hold for $z\in X'_{reg}$. 
\end{proposition}

If $\phi\in\W_{0,0}(X)$ and $\dbar_X\phi=0$,  thus $\phi\in\Ok(X)$. If 
$\phi\in\W_{0,q+1}(X)$,  and $\debar_X \phi=0$, then there is $u\in\W_{0,q}(X')$
such that $\dbar u=\phi$ in $X'_{reg}$.
%%%
However, we do not know whether $\K \phi$, or any other $u$ in $\W^X_{0,*}$ that solves $\debar u = \phi$ on $X_{reg}$,
is in $\mbox{Dom}\,\debar_X$ in general; if $\mbox{dim}\, X=1$ then $\K \phi$ is
indeed in $\mbox{Dom}\,\debar_X$, see Section~\ref{kurvsektion} below.
However, we can show that the difference of two such solutions is in fact (strongly) holomorphic if $q=1$
and $\debar$-exact on $X_{reg}$ if $q>1$. By an elaboration of this kind of arguments we can get
a global version of Proposition~\ref{Wprop}, see Theorem~1.8 in \cite{AS}.

%%After establishing Proposition \ref{Wprop}, the crucial step in the proof of 
%%Theorems~\ref{main} and \ref{koppelman} is to verify  that  
%%$\A^X$ is a subsheaf  of $\mbox{Dom}\, \debar_X$; this is the content of Lemma~\ref{mainlemma}
%%below.

\begin{center}
---
\end{center}

\noindent Let us mention a few further applications of the Koppelman formulas \eqref{Koppel1} and \eqref{Koppel2}.
%%
%%They  have recently been used in \cite{ASS} to give a semi-explicit analytic proof of Huneke's uniform
%%Brian\c{c}on-Skoda theorem, \cite{Hu}, in the case of ideals of the local ring $\hol^X_x$. 
Let $X$ and $X'$ be as in Theorem~\ref{koppelman} and let $\delta(z)$ be the distance to $X_{sing}$. 
There is an integer $K$ such that
if $\phi\in C^{k}_{0,q}(X)$, $k\geq K+1$, $q\geq 1$ and  $\debar \phi=0$ (on $X_{reg}$), then
$u=\K\phi$ has meaning,
is in $C^{k}_{0,q}(X'_{reg})$,  $\debar u=\phi$ on $X'_{reg}$ and $u(z)=\ordo(\delta(z)^{-M})$
for some integer $M$ not depending on $\phi$, see \cite{AS}.
Furthermore, 
if $\phi \in C^{K+1}_{0,0}(X)$ is $\debar$-closed on $X_{reg}$, then $\phi\in \hol(X)$; 
this result was proved in  \cite{Mal} and \cite{Spallek}, but our formula gives an explicit
holomorphic extension $\Proj\phi$ to ambient space $\Omega'$.
%%\begin{remark}
%%With a possibly larger choice for $K$ we actually have that our solution $u=\K\phi$ extends to a current $U$
%%such that $\debar (U\wedge \omega) = \phi\wedge \omega$ holds on $X'$, 
%%where the product $U\wedge \omega$ is defined in a natural way. 

One can incorporate extra weight factors into the Koppelman formulas. 
In \cite{AS} we use this techique  to obtain  results in the spirit of \cite{FOV}. For instance, 
%%letting $\delta$ be the distance to $X_{sing}$
we prove that given $M\geq 0$ there is an $M'\geq 0$ such that if $\phi$ is a $\debar$-closed $(0,q)$-form, $q\geq 1$, 
on $X_{reg}$ with $\delta^{-M'}\phi\in L^p(X_{reg})$, $1\leq p\leq \infty$, then there is a $(0,q-1)$-form $u$
on $X'_{reg}$ with $\debar u=\phi$ there and $\delta^{-M} u\in L^p(X_{reg})$.
In \cite{AS} we also use extra weight factors  to solve the 
$\debar$-equation with compact support. As a consequence we relate extension results to
the cohomological dimension $\nu$ of $X$. For instance, if $\nu\geq 2$, then the Hartog's extension phenomenon holds.
Moreover, if $\phi$ is a smooth $\debar$-closed $(0,q)$-form on $X_{reg}$ and $q<\nu-1-\mbox{dim}\, X_{sing}$, then
there is a smooth solution to the $\debar$-equation on $X_{reg}$; this was proved by Scheja by
purely cohomological methods in the early '60s, \cite{Scheja}.

\begin{center}
---
\end{center}

\noindent Let us also mention a few other works related to the $\debar$-equation on complex spaces. In particular, the 
$L^2$-$\debar$-cohomology of (the regular part of) complex spaces 
(usually with respect to the non-complete metric given by embeddings) has been investigated in, e.g., 
\cite{Ohsa}, \cite{Nag}, \cite{PaSt1},
\cite{PaSt2}, \cite{DFV}, \cite{OvVass}, \cite{Rupp2}. Also H\"{o}lder and $L^p$-estimates have been studied, 
first in \cite{FoGa} and later, using integral formulas, in, e.g., \cite{Rupp}, \cite{RZ}.

\smallskip 

We have organized the paper as follows. In Section~\ref{PMsektion} we recall the basic facts from \cite{AW2} about
{\em pseudomeromorphic} currents and introduce a slightly more general notion thereof. We also define
the sheaf $\W^X$. 
Following the ideas in \cite{AW1}, we introduce in Section~\ref{Rsektion}
a certain residue  current $R$ associated to a subvariety of $\C^N$ as well as the associated
intrinsic structure form $\omega$ on $X$ that is  fundamental for the definition
of the strong $\dbar$-operator,
$\dbar_X$,  in Section~\ref{starksektion} and for  the construction of the Koppelman formulas.
%%In particular, the fundamental Proposition~\ref{fundprop} is proved except for the asymptotic estimate, 
%%which is postponed until Section~\ref{uppsksektion}. 
In Section~\ref{KoppelsektionI} we construct
such formulas on a subvariety of a pseudoconvex domain $\Omega \subset \C^N$ 
by using a Koppelman formula in $\Omega$ and ``concentrating''
it to the subvariety by means of the current $R$. We use these results in 
Section~\ref{KoppelsektionII}  to obtain  
intrinsic semi-global Koppelman formulas on an arbitrary pure dimensional reduced complex space.
We also prove Proposition~\ref{Wprop}.  
In  Section~\ref{asec}  we define the sheaf $\A^X$ and prove our main theorem.
In Section~\ref{exsektion} we provide a few examples that illustrates our construction. As special cases we get back the 
representation formulas of Stout and Hatziafratis in \cite{Stout} and \cite{Hatzi}
for (strongly) holomorphic functions.
In the case when $X$ is a complex curve we strengthen our results in Section~\ref{kurvsektion}.

\smallskip
\noindent{\bf Acknowledgement:} We are grateful to the referee for careful reading and many important
comments and suggestions.

\section{The sheaves $\PM^X$ and  $\W^X$}\label{PMsektion}

As in the smooth case, the sheaf  $\Curr^X_{p,q}$ of currents
of bidegree $(p,q)$ on $X$ is by definition  the dual of $\D^X_{n-q,n-p}$. 
Given a local embedding $i\colon X\hra \Omega\subset\C^N$, thus, the currents $\mu\in \Curr^X_{p,q}$
precisely correspond, via $\mu \mapsto i_* \mu$, to the currents of bidegree $(N-n+p,N-n+q)$ in the ambient space
that vanish on all test forms $\xi$ such that $i^*\xi =0$ on $X_{reg}$. 
If $u$ and $\mu$ are currents on $X$ we say that
$\debar u=\mu$ if  $u.\,\debar \xi = \pm \mu. \xi$
for all test forms $\xi$ on $X$. This is equivalent to that $\debar (i_*u)=i_*\mu$ in the ambient space.

Recall that in one complex variable $x$ one can define  the principal value current $1/x^m$, $m\ge 1$, 
as the value at $\lambda=0$ of the analytic continuation of $|x|^{2\lambda}/x^m$. Then the residue current
$\dbar(1/x^m)$ is the value at $\lambda=0$ of $\dbar|x|^{2\lambda}/x^m$; clearly it has its support
at $x=0$. 
Assume now that $x_j$ are holomorphic coordinates on $\C^N$. 
Since we can take tensor products of one-variable currents, we can form
the current %%%  (cf., also Remark~\ref{chprod} below)
\begin{equation*}
\tau=\debar \frac{1}{x_1^{a_1}}\wedge \cdots \wedge \debar \frac{1}{x_r^{a_r}}\wedge 
\frac{\gamma(x)}{x_{r+1}^{a_{r+1}}\cdots x_N^{a_N}},
\end{equation*}
where $a_1,\ldots,a_r$ are positive integers, $a_{r+1},\ldots,a_N$ are nonnegative
integers, and  $\gamma$ is a smooth compactly supported form. Such a $\tau$
is called an {\em elementary} pseudomeromorphic current.  It is  commuting
in the principal value factors and anti-commuting in the residue factors.

\begin{definition}
We say that a current $\mu$ on $X$ is {\em pseudomeromorphic}, $\mu\in \PM(X)$, if it is 
locally a finite sum of pushforwards $\pi_*\tau=\pi_*^1 \cdots \pi_*^m \tau$ under maps
\begin{equation}\label{PMdef}
X^m \stackrel{\pi^m}{\longrightarrow} \cdots \stackrel{\pi^2}{\longrightarrow} X^1 
\stackrel{\pi^1}{\longrightarrow} X^0=X,
\end{equation}
where each $\pi^j\colon X^j \to X^{j-1}$ is either a modification,  a simple projection
$X^{j-1}\times Z \to X^{j-1}$, or an open inclusion (i.e., $X^{j}$ is an open subset
of $X^{j-1}$), and $\tau$ is elementary on $X^m$.
\end{definition}

\begin{remark}\label{pmvek}
The notion of pseudomeromorphic currents on a complex manifold was introduced in \cite{AW2} and
the same definition works on  our  singular space $X$. The definition in \cite{AW2}
is however somewhat more restrictive as it  does not allow  simple projections.
As we will see,
all basic properties are preserved also for our class. 
\end{remark}

It is clear that the \pmm currents on $X$ define a  subsheaf $\PM^X$ of $\Curr^X$.
It is 
closed under $\debar$ and under exterior products with smooth forms since 
this is true for (finite sums of) elementary currents. Moreover, if $\mu\in \PM(X)$ and $V\subset X$ 
is an analytic subvariety, there is a unique pseudomeromorphic current $\mathbf{1}_{X\setminus V}\mu$
on $X$, the so-called {\em standard extension} of the natural restriction of $\mu$ to the open 
set $X\setminus V$. If $h$ is a holomorphic tuple such that $V=\{h=0\}$, then 
$|h|^{2\lambda}\mu$ (a priori defined for $\real \,\lambda \gg 0$) has a current valued analytic continuation
to $\real \, \lambda >-\epsilon$. The value at $\lambda=0$ is precisely the current $\mathbf{1}_{X\setminus V}\mu$,
i.e., we have
\begin{equation}\label{restrikdef}
\mathbf{1}_{X\setminus V}\mu = |h|^{2\lambda} \mu |_{\lambda =0}.
\end{equation} 
We can also obtain $\mathbf{1}_{X\setminus V}\mu$ as a principal value: If $\chi$ is a smooth approximand
of the characteristic function of $[1,\infty)$ on $\R$, then 
\begin{equation}\label{restrikdef2}
\mathbf{1}_{X\setminus V}\mu = \lim_{\delta\to 0^+} \chi(|h|/\delta) \mu.
\end{equation}
It follows that the current 
\begin{equation*}
\mathbf{1}_V \mu :=\mu-\mathbf{1}_{X\setminus V}\mu
\end{equation*}
is in $\PM^X$ and has support on $V$, and in particular that %%Moreover, cf.,  \eqref{restrikdef},
$\1_V\mu=\mu$ if $\mu$ has support on $V$.
The existence of \eqref{restrikdef} and the independence of $h$ follow from the corresponding statements 
for elementary currents as in \cite{AW2}, noting that if $\mu =\pi_* \tau$, then 
\begin{equation}\label{boll}
|h|^{2\lambda} \mu =\pi_* (|\pi^* h|^{2\lambda}\tau)
\end{equation}
for $\real \, \lambda \gg 0$. In the same way one can reduce the verification of \eqref{restrikdef2} to the case
with elementary currents.

These operations  satisfy the computation rule
\begin{equation}\label{rakneregel}
\mathbf{1}_V \mathbf{1}_{V'} \mu = \mathbf{1}_{V\cap V'} \mu.
\end{equation} 
In fact, if $\mu =\pi_*\tau$, then, cf.,  \eqref{boll},
\begin{equation}\label{boll2}
\mathbf{1}_V \mu = \pi_*(\mathbf{1}_{\pi^{-1}(V)} \tau),
\end{equation}
and so \eqref{rakneregel} follows from the corresponding statement for the more restricted 
pseudomeromorphic currents in \cite{AW2}, cf., Remark~\ref{pmvek}. 

If $\mu\in \PM(X)$, then $\mu \otimes 1 \in \PM(X\times Y)$. To see this, 
assume that $\mu=\pi_* \tau$,
where $X^m \stackrel{\pi}{\longrightarrow} X$ is a composition of modifications,  simple projections,
and open inclusions, 
cf., \eqref{PMdef}, and $\tau$ is elementary on $X^m$. Let $Y' \stackrel{\pi'}{\longrightarrow} Y$ be a 
modification with $Y'$ smooth and let $\{\rho_j\}_j$ be a locally finite partition of unity on $Y$.
The map $\Pi\colon X^m\times Y' \to X\times Y$ defined by $\Pi(x,y)=(\pi(x),\pi'(y))$ now is a composition of 
modifications,  simple projections, and open inclusions,  $\tau\otimes \Pi^*\rho_j$ are elementary on $X^m\times Y'$, 
and $\mu\otimes 1 =\sum_j \Pi_* (\tau\otimes \Pi^*\rho_j)$.

%%The following properties of pseudomeromorphic currents are crucial and will be used frequently.

\begin{proposition}\label{grad/stod}
Assume that $\mu\in \PM(X)$ has support on the subvariety $V\subset X$.

\begin{itemize}  
\item[(i)] If $\xi\in\Ok(X)$ is vanishing on $V$, then
\begin{equation}\label{bolt}
\bar{\xi} \mu=0, \quad d\bar{\xi}\wedge \mu = 0.
\end{equation}

\item[(ii)] \rm{(Dimension\ principle).}\  
If $\mu$ has bidegree $(*,p)$ and $\mbox{codim}\, V > p$, then $\mu=0$.
\end{itemize}

\end{proposition}

The dimension principle is crucial and will be used frequently.

\begin{proof}
%%We will reduce the proof to the corresponding one in \cite{AW2} for the more restrictive definition of $\PM$.
Assume that $\mu=\pi_* \tau$, where $X^m \stackrel{\pi}{\longrightarrow} X$ is a composed map as in
\eqref{PMdef} and let $V^m:=\pi^{-1}V$.
If $\mu$ has support on $V$, then,  cf., \eqref{boll2}, 
$\mu=\1_V\mu=\pi_*(\mathbf{1}_{V^m}\tau)$,  so $\mu$ is the pushforward of
$\tilde\mu:=\1 _{V^m}\tau$, which is \pmm  in the sense of \cite{AW2} and has support on
$V^m$. If $\xi\in\Ok(X)$ vanishes on $V$, then $\pi^* \xi$ vanishes on $V^m$ and from \cite{AW2} we then know
that $\pi^*(\bar{\xi}) \tilde{\mu}=0$ and  $d\pi^*(\bar{\xi})\wedge \tilde{\mu}=0$. Hence, \eqref{bolt}
holds and so (i) is proved. %%$\bar{\xi} \mu=0$ and $d\bar{\xi}\wedge \mu = 0$.

Now, choose a local embedding $i \colon X\hra \Omega\subset \C^N$. If the hypotheses
of (ii)  hold, then $i_* \mu$ is a current in $\Omega$ of bidegree $(N-n+*,N-n+p)$ that has support on 
the variety $i(V)$, which has codimension $>N-n+p$. Moreover, $d\bar{\eta}\wedge i_*\mu=0$
for every holomorphic function $\eta$ vanishing on $i(V)$ since
\begin{equation*}
d\bar{\eta}\wedge i_*\mu = i_*\big( d(i^*\bar{\eta})\wedge \mu \big) = 0
\end{equation*}  
in view of part (i). Arguing as in the proofs of Theorem~2.10 and Corollary~2.11 in Chapter~III of \cite{JP},
it now follows that $i_*\mu=0$, and hence $\mu=0$.
\end{proof}

We say that $\mu\in\PM(X)$ has the {\em standard extension property}, SEP, on $X$ if $\mathbf{1}_V \mu=0$ for
every subvariety $V\subset X$ of positive codimension.

\begin{definition}[The sheaf $\W^X$]
Let $X$ be a complex space. We define  $\W^X$ as the subsheaf of $\PM^X$ of currents with the SEP on $X$.
%%We denote by $\W^X_{0,q}$ the currents in $\W^X$ of bidegree $(0,q)$.
\end{definition}

It is easy to check that 
\begin{equation}\label{skata}
\1_V (\phi\w \mu)=\phi\w \1_V\mu
\end{equation}
if $\phi$ is a smooth form  and  $\mu$ 
is pseudomeromorphic.  It follows that 
the  sheaf $\W^X$ is closed under multiplication by smooth forms.
%%%
If $\pi\colon Y\to X$ is either a modification or a  simple projection, 
 then $\pi_*$ maps
$\W(Y)\to \W(X)$. This follows from \eqref{boll2} since $\pi^{-1}(V)$ has positive codimension
if $V$ has positive codimension. %%; notice that it is crucial that $\pi$ is surjective. 

\begin{center}
---
\end{center}

\noindent Let $f$ be holomorphic on $X$. Then $|f|^{2\lambda}/f$  has an analytic continuation 
to $\real\, \lambda>-\epsilon$ and the value at $\lambda=0$ is a \pmm current,
that we denote $1/f$, which  coincides
with the function $1/f$ where $f$ is nonvanishing, see, e.g., \cite{hasamArkiv}. 
We call it the {\it principal value current}  $1/f$ associated with the meromorphic 
function\footnote{Somewhat abusively we often identify the meromorphic function
and the associated current.} $1/f$.
It follows from the dimension principle that $1/f$ has the SEP on $X$ and hence
is in $\W^X$.
The product of such a current
and a smooth form is called a {\it semi-meromorphic current} (or {\it form})
in $X$. 
%%In view of
%%\eqref{skata},  each semi-meromorphic current is in $\W^X$.
It will be convenient to introduce a more general notion.

\begin{definition}\label{semidef}
We say that a current $\alpha$ on a complex space $X$ is {\em almost semi-meromorphic} if 
there is a smooth 
modification $p\colon \tilde{X}\to X$ and a semi-meromorphic 
current $\tilde{\alpha}$ on $\tilde{X}$ such that $p_* \tilde{\alpha}=\alpha$.
\end{definition}

Notice that almost semi-meromorphic currents (forms)  are in $\W^X$.
One can also verify that a current $\alpha$ on $X$ is almost semi-meromorphic if and only if 
there is an analytic set $V\subset X$ of positive codimension such that $\alpha$ is smooth on $X\setminus V$,
and a smooth modification $p\colon \tilde{X}\to X$, inducing a biholomorphism 
$\tilde{X}\setminus p^{-1}(V)\to X\setminus V$, such that $(p|_{X\setminus V})^* \alpha$ extends to a semi-meromorphic
current on $\tilde{X}$.

It follows that  if $\alpha$ is almost semi-meromorphic in  $X$ and $\pi\colon X'\to X$ is a smooth
modification, then the pullback $\pi^*\alpha$ is well-defined and  again almost
semi-meromorphic.  
%%One can also verify that finite sums of almost semi-meromorphic currents
%%are almost semi-meromorphic.

\begin{remark}
Let us say that  $\alpha$ is {\it locally} (almost) semi-meromorphic if for each point there is a \nbh where
$\alpha$ is (almost) semi-meromorphic.
For instance, if $f$ is holomorphic in $\Omega\subset\C^N$
and $\gamma$ is a test form with support in $\Omega$, then the principal value current
$\gamma/f$ is semi-meromorphic in $\Omega$  and locally semi-meromorphic in $\C^N$. 
  One can check that a current $\alpha$ on $X$ is locally 
almost semi-meromorphic if and only if $\alpha$ is 
pseudomeromorphic in the sense of \cite{AW2}, cf., Remark~\ref{pmvek},
and has the SEP. In  fact, fix a point $x$ and suppose that $\alpha$ is a finite
sum of currents $\pi_*\tau_\ell$ as in Remark~\ref{pmvek} in a \nbh of $x$.
Let $V$ be the union of the images $\pi(V_\ell)$, where 
for each $\ell$, $V_\ell$ is the affine subspace where the
product of the residue  factors in $\tau_\ell$ has its support. Then $V$
is a germ of  an analytic subset of $X$ at $x$ of positive codimension and hence
by the SEP, $\1_V\alpha=0$. Thus $\alpha$ is the push-forward of 
elementary currents with no residue factors, and hence it is
almost semi-meromorphic at $x$.
\hfill $\Box$
\end{remark}

%%The sheaf $\W^X$ is obviously closed under exterior products with smooth forms but we actually have much more
%%freedom. 

\begin{proposition}\label{multprop}
Let $\alpha$ be  almost semi-meromorphic on $X$. For each $\mu\in \W(X)$, the current $\alpha\wedge \mu$,
a priori defined where $\alpha$ is smooth, has a unique extension to a current in $\W(X)$. 
\end{proposition}

It is natural to let $\alpha\wedge \mu$ denote this extension as well.

\begin{proof} %%Since the statement is local, we may assume that $\alpha$ is semi-meromorphic on $X$.
For the uniqueness of an extension of $\alpha\wedge \mu$ in $\W(X)$ we only have to observe that two currents with 
the SEP on $X$, that coincide outside an analytic set of positive codimension, must be equal everywhere.

Let $V=\{h=0\}\subset X$ be an analytic subset of positive codimension such that $\alpha$ is smooth in $X\setminus V$.
%% and $V$ does not contain any irreducible component of $X$ where $\alpha$ is generically
%%non-vanishing. 
We claim that the analytic continuation
\begin{equation}\label{ball}
|h|^{2\lambda}\alpha \wedge \mu |_{\lambda=0}
\end{equation}
exists and defines a current $\nu \in \PM(X)$. It is clearly a pseudomeromorphic extension of 
$\alpha\wedge \mu$ from $X\setminus V$ to $X$, and since $\mu$ has the SEP, it follows that 
$\mathbf{1}_A \nu = 0$ in $X\setminus V$ for any analytic subset $A\subset X$ with positive codimension.
Moreover, 
\begin{eqnarray*}
\mathbf{1}_{X\setminus V}\nu &=& |h|^{2\lambda'}\nu|_{\lambda'=0} =
\big(|h|^{2\lambda'} |h|^{2\lambda} \alpha \wedge \mu |_{\lambda=0}\big)|_{\lambda'=0} \\
&=&
|h|^{2\lambda''} \alpha\wedge \mu |_{\lambda''=0} =\nu,
\end{eqnarray*}
and in view of \eqref{rakneregel} we can conclude that $\nu$ has the SEP in $X$ and thus is in  $\W(X)$.

To prove the claim, let $\mu=\pi_* \tau$,
where $\pi\colon X^m \to X$ is a composition of modifications,  simple projections, and open inclusions 
as in \eqref{PMdef}, and $\tau$ is elementary on $X^m$.
Let $p\colon \tilde{X}\to X$ be a smooth modification and let $\tilde{\alpha}$ be a semi-meromorphic current on 
$\tilde{X}$ such that $p_* \tilde{\alpha}=\alpha$.
By a standard argument  \eqref{PMdef} can be extended to a commutative diagram 
\begin{equation}\label{kommdiagram}
\begin{array}[c]{ccccccccc}
\tilde{X}^m & \stackrel{\tilde{\pi}^m}{\longrightarrow} & \cdots & \stackrel{\tilde{\pi}^2}{\longrightarrow} & \tilde{X}^1 & \stackrel{\tilde{\pi}^1}{\longrightarrow} & \tilde{X}_0 & = & \tilde{X} \\
\downarrow \scriptstyle{p_m} & & & & \downarrow \scriptstyle{p_1} & & \downarrow \scriptstyle{p} & & \\
X^m & \stackrel{\pi^m}{\longrightarrow} & \cdots & \stackrel{\pi^2}{\longrightarrow} & X^1 & \stackrel{\pi^1}{\longrightarrow} & X_0 & = & X
\end{array}
\end{equation}
so that each vertical map is a modification and each $\tilde{\pi}^j\colon \tilde{X}^j\to \tilde{X}^{j-1}$
is either a modification, a simple projection, or an open inclusion. 
%%To this end, assume that this is done up to level $k$.
%%It is well-known that if $\pi^{k+1}\colon X^{k+1}\to X^{k}$ is a modification, then there are 
%%modifications $\tilde{\pi}^{k+1}\colon \tilde{X}^{k+1}\to \tilde{X}^{k}$ and 
%%$p^{k+1}\colon \tilde{X}^{k+1}\to X^{k+1}$ such that 
%%\begin{equation*}
%%\begin{array}[c]{ccc}
%%\tilde{X}^{k+1} & \stackrel{\tilde{\pi}^{k+1}}{\longrightarrow} & \tilde{X}^k \\
%%\downarrow \scriptstyle{p_{k+1}} & & \downarrow \scriptstyle{p_k} \\
%%X^{k+1} & \stackrel{\pi^{k+1}}{\longrightarrow} & X^k
%%\end{array}
%%\end{equation*}
%%commutes. If instead $X^{k+1}=X^k \times Z$ then we simply take
%%$\tilde{X}^{k+1}=\tilde{X}^k \times Z$. 
Now, let $\tilde{\pi}$ denote the composition of the maps in the 
first row of \eqref{kommdiagram}. 
Then $\hat\alpha:=\tilde{\pi}^* p^* \alpha=\tilde\pi^*\tilde\alpha$ is semi-meromorphic 
on $\tilde{X}^m$ and moreover, there is an elementary (or at least a pseudomeromorphic 
in the sense of \cite{AW2}) current $\hat\tau$ on $\tilde{X}^m$
such that $(p_m)_* \hat\tau = \tau$. 
It follows from Proposition~2.1 in \cite{AW2} that 
\begin{equation*}
\lambda \mapsto |\tilde{\pi}^* p^* h |^{2\lambda} \hat\alpha \wedge \hat\tau
\end{equation*}
has an analytic continuation to $\real \, \lambda > -\epsilon$ and that the value at $\lambda=0$ defines a 
pseudomeromorphic current $\hat\nu$ on $\tilde{X}^m$. Thus, the analytic continuation
\eqref{ball} exists and 
$\nu = |h|^{2\lambda}\alpha \wedge \mu |_{\lambda=0} = p_*\tilde{\pi}_*  \hat\nu \in \PM(X)$.
\end{proof}

\begin{example}\label{ba}
Let $W$ be a hypersurface in $X$.
We claim that if  $\alpha\in \PM_{k,0}(X)$ and  the restriction $\alpha'$ 
to $X\setminus W$ is holomorphic,   then  $\alpha$ is meromorphic on $X$.
In fact, by assumption, $\alpha'$ has a current extension to $X$, so 
if we have an embedding  $i \colon X\hra \Omega$, then the current $i_*\alpha'$, 
a~priori defined in $\Omega\setminus W$,  has a current extension to $\Omega$.
By  \cite[Theorem~1]{HePa},  $\alpha'$ has
a meromorphic extension $\tilde\alpha$ to $X$, and since both $\alpha$ and $\tilde\alpha$
are in $\PM^X_{k,0}$,   $\alpha=\tilde\alpha$  by the dimension principle.

Let $a$ be a meromorphic form in $\Omega$  such that $\alpha=i^*a$. Then
$i_*\alpha=a\w[X]$, where $[X]$ is the Lelong current associated with $X$ in $\Omega$,
so
$\dbar \alpha=0$ on $X$ precisely means that
$\dbar(a\w[X])=0$ in $\Omega$. 
This in turn by the definition in \cite{HePa} means that 
$\alpha$ is in the sheaf (that we denote) $\Ba^X_k$
of Barlet-Henkin-Passare holomorphic $(k,0)$-forms.
%%%
We conclude that $\Ba^X_k$ is the subsheaf of $\dbar$-closed currents in $\PM^X_{k,0}$.
The sheaves  $\Ba^X_k$ were defined in this way in \cite{HePa}
but introduced earlier by  Barlet, \cite{Barlet}, in a different  way, 
see  \cite[Remark~5]{HePa}.  
\end{example}

\begin{example}\label{chprod}
Let $a_1,\ldots,a_p$ be holomorphic on $X$. Following \cite{AW2} we can inductively define
the \pmm currents
$$
\mu_{k+1}=\dbar\frac{1}{a_{k+1}}\w\dbar\frac{1}{a_k}\w\dots\w\dbar\frac{1}{a_1}:=
\frac{\dbar|a_{k+1}|^{2\lambda}}{a_{k+1}}\w\dbar\frac{1}{a_k}\w\dots\w\dbar\frac{1}{a_1}\Big|_{\lambda=0}.
$$
Then  $\mu_p$ has support on the set $V=\{a_p=\cdots=a_1=0\}$.
If $X$ is smooth and $\codim V=p$, then $\mu_p$  is the so-called {\it Coleff-Herrera product};
see \cite{BS} and \cite{LS}  for a thourough discussion of various possible definitions.
\end{example}

\section{The notion of  structure form $\omega$ on $X$}\label{Rsektion}

To begin with, let $i\colon X \hra \Omega$ be a (reduced) hypersurface  in a pseudoconvex domain
$\Omega\subset\C^{n+1}$, i.e., $X=\{f=0\}$ where $f$ is holomorphic in $\Omega$ and $df\neq 0$
on $X_{reg}$.  If $\omega'$ is a meromorphic $(n,0)$-form in $\Omega$ 
(or in a small \nbh of $X$ in $\Omega$)
such that
\begin{equation}\label{lerayform}
(df/2\pi i) \w \omega'=d\zeta_1\w\ldots\w d\zeta_{n+1}
\end{equation}
on $X$, then $\omega:=i^*\omega'$ is a meromorphic form on $X$ that is independent of the choice of
$\omega'$, and the classical Leray  residue formula states that for test forms $\psi$ of bidegree
$(0,n-1)$, the principle value integral
$$
\int_X\omega\w i^*\psi
$$
is equal to the action of the residue current $\dbar(1/f)$ on  the test form 
$\psi d\zeta_1\w\ldots\w d\zeta_{n+1}$.
%%where $d\zeta:=d\zeta_1\w\ldots\w d\zeta_{n+1}$.
This equality can be rephrased as 
\begin{equation}\label{leray}
i_*\omega=\dbar\frac{1}{f}\w d\zeta_1\w\ldots\w d\zeta_{n+1}.
\end{equation}
If $\partial f/\partial \zeta_{n+1}$ is not vanishing identically on (any irreducible component of)
$X$, one can take, e.g., $\omega'=1/(\partial f/\partial \zeta_{n+1})d\zeta_1\w\ldots\w d\zeta_n$. 
Notice that under this assumption on $f$, 
any meromorphic form on $X$ can be written $h d\zeta_1\w\ldots\w d\zeta_n$ for a unique
meromorphic function $h$.
%%The formula \eqref{leray} follows quite immediately from the Poincar\'e-Lelong formula
%%$\dbar(1/f)\w df/2\pi i=[X]$, where $[X]$ is the Lelong current in $\Omega$ associated with  $X$.
%%
It follows from \eqref{leray} that $\dbar\omega=0$ so $\omega$ 
is in $\Ba_n(X)$, cf., Example~\ref{ba}.  The form $\omega$  also has the
following two properties:

\smallskip
\noindent$(i)$ If $\phi$ is a meromorphic function  on $X$, then
$\phi$ is in $\Ok^X$ if (and only if)   $\dbar(\phi\omega)=0$.

\smallskip
\noindent$(ii)$  If $\alpha$ is in $\Ba^X_n$ then $\alpha= h\omega$ for some $h$ in $\Ok^X$.

\smallskip
%%Thus $(i)$ gives a criterion  for a meromorphic function to be holomorphic,
%%whereas  
%%
%%$(ii)$ means that $\omega$ is a generator for the $\Ok$-module $\Ba^X_n$.
%%
\noindent Since any  meromorphic $(n,0)$-form $\alpha$ is 
$h\omega$ for a unique meromorphic function $h$, 
$(i)$ and $(ii)$ are in fact equivalent; for a proof of $(i)$, see, e.g.,
\cite[Remark~3]{HePa} or below.

%%\marginpar{hitta system i specialindex for intrinsica saker!!}

\medskip
For the rest of this section let $i\colon X \hookrightarrow \Omega\subset\C^N$
be a pure $n$-dimensional subvariety of the pseudoconvex domain $\Omega$, and let
$p:=N-n$ be its codimension. 
We will introduce an almost semimeromorphic form  $\omega$ on $X$,  that satisfies
an analogue of \eqref{leray}. %%We call it  a {\it structure form} on  $X$. 
In a reasonable sense it will also fulfill  $(i)$ and $(ii)$. 
It can be noted that $\omega$ plays a central role in \cite{ASS}.
To begin with we look for  an 
adequate generalization of the residue current $\dbar(1/f)$. 

If $f$ is any holomorphic function, then 
a holomorphic function $\phi$ is in the ideal $(f)$ generated by $f$ if and only if
$\phi\dbar(1/f)=0$. 
Given any ideal sheaf $\J$ in $\Omega$, in \cite{AW1}  was  constructed a residue current
$R$ such that 
\begin{equation}\label{Rduality}
\phi\in\J \text{\ if\ and\ only\ if } \   \phi R=0 
\end{equation}
if $\phi\in\Ok^\Omega$.
It is thus reasonable to consider  $R$ when $\J=\J_X$ is
the radical ideal sheaf $\J=\J_X$  associated with $X$, so let us first recall its
definition.
In a slightly smaller set, still denoted $\Omega$, there is  a free
resolution
\begin{equation}\label{karvupplosn}
0 \to \hol(E_m) \stackrel{f_m}{\longrightarrow} \cdots \stackrel{f_3}{\longrightarrow} \hol(E_2)
\stackrel{f_2}{\longrightarrow} \hol(E_1) \stackrel{f_1}{\longrightarrow} \hol(E_0) 
\end{equation}
of the sheaf $\hol^{\Omega}/\J_X$; 
here  $E_k$ are trivial vector bundles over $\Omega$ and $E_0\simeq \C\times \Omega$ is the
trivial line bundle. This resolution induces a complex of  vector bundles
\begin{equation}\label{VBkomplex}
0 \to E_m \stackrel{f_m}{\longrightarrow} \cdots \stackrel{f_3}{\longrightarrow} E_2
\stackrel{f_2}{\longrightarrow} E_1 \stackrel{f_1}{\longrightarrow} E_0
\end{equation}  
that is pointwise exact outside $X$. Let $X_k$ be the set where $f_k$ does not have optimal rank. Then
\begin{equation*}
\cdots \subset X_{k+1} \subset X_k \subset \cdots \subset X_{p+1}\subset X_{sing}\subset X_p =\cdots =X_1= X;
\end{equation*}
these sets are independent of the choice of resolution and  thus invariants of the sheaf $\F:=\hol^{\Omega}/\J_X$.
%%The Buchsbaum-Eisenbud theorem claims that $\mbox{codim}\, X_k \geq k$ for all $k$, and since furthermore
Since $\F$ has {\it pure} codimension $p$ (i.e., no embedded prime ideals), 
%%in our case, $X_k\subset X_{sing}$, and 
\begin{equation}\label{BEsats}
\mbox{codim}\, X_k \geq k+1, \quad \mbox{for} \quad k\geq p+1,
\end{equation}
see Corollary 20.14 in \cite{Eis}. There is a free resolution \eqref{karvupplosn} 
if and only if $X_k=\emptyset$ for $k>m$.
Thus we can always have $m\le N-1$. 
%%The smallest number $m$ such that $X_k=\emptyset$ for $k>m$ is equal to 
%%$N-\nu$, where $\nu$ is the minimal depth (cohomological dimension) of $\F$.
The variety is Cohen-Macaulay, i.e., the sheaf $\F$ is Cohen-Macaulay, if and only if 
$X_k=\emptyset$ for $k\geq p+1$. In this case we can thus choose a resolution \eqref{karvupplosn}
with $m=p$. 
If we define
\begin{equation}\label{oskar}
X^0=X_{sing}, \quad X^r=X_{p+r}, \,\, r\geq 1,
\end{equation} 
then 
\begin{equation*}
X^{n-1}\subset\cdots \subset X^1\subset X^0\subset X,\quad \quad  \codim X^k\ge k+1.
\end{equation*}
The sets $X^k$ are  independent 
of the choice of embedding, see \cite[Lemma~4.2]{AWsemester},
and  are thus intrinsic subvarieties of the complex space $X$ 
and   reflect  the complexity  of the singularities of $X$.

\medskip

\noindent Let us now choose Hermitian metrics on the bundles $E_k$. We then refer to
\eqref{karvupplosn} as a {\it Hermitian free resolution} of $\Ok^X/\J_X$ in $\Omega$. 
In $\Omega\setminus X_k$ we have a well-defined vector bundle
morphism $\sigma_{k+1}\colon E_k\to E_{k+1}$, if we require that  $\sigma_{k+1}$ vanishes on
$(\Image f_{k+1})^\perp$, takes values in $(\Ker f_{k+1})^\perp$ and that $f_{k+1}\sigma_{k+1}$ is
the identity on $\Image f_{k+1}$.   Following  \cite{AW1} we  define the smooth 
$E_k$-valued forms
\begin{equation}\label{ukdef}
u_k=(\dbar\sigma_k)\cdots(\dbar\sigma_2)\sigma_1=\sigma_k(\dbar\sigma_{k-1})\cdots(\dbar\sigma_1)
\end{equation}
in $\Omega\setminus  X$; for the second equality, see \cite[(2.3)]{AW1}. We have that 
$$
f_1u_1=1, \quad f_{k+1}u_{k+1}-\dbar u_k=0,\quad  k\ge 1,
$$
in $\Omega\setminus  X$.
If $f:=\oplus f_k$ and $u:=\sum u_k$,  then these relations can be written
economically as $\nabla_f u=1$ where
$\nabla_f:= f-\dbar$.
To make the algebraic machinery  work properly one has to introduce a superstructure on the
bundle $E=:\oplus E_k$ so that vectors in $E_{2k}$  are even and vectors in  $E_{2k+1}$ are odd,
and hence $f$, $\sigma:=\oplus\sigma_k$,  and  $u:=\sum u_k$  are odd.
For details, see \cite{AW1}. 
It turns out that $u$ has a (necessarily unique)  
almost semi-meromorphic extension $U$ to $\Omega$,  and the current
$R$ is  defined by the relation
$$
\nabla_f U=1-R.
$$
If $F$ is any holomorphic tuple that vanishes on $X$, then
\begin{equation}\label{rdef}
U=|F|^{2\lambda}u|_{\lambda=0}, \quad R=\dbar|F|^{2\lambda}\w u|_{\lambda=0}.
\end{equation}
Thus  $R$ has support on $X$ and is a sum $\sum R_k$, where $R_k$ is
a \pmm $E_k$-valued current of bidegree $(0,k)$. It follows from  the dimension
principle that $R=R_p+R_{p+1}+\cdots +R_N$. Since we can always choose a resolution that
ends at level $N-1$, cf., \eqref{BEsats},  we may assume that $R_N=0$. 
If $X$ is Cohen-Macaulay and $m=p$ in \eqref{karvupplosn}, then $R=R_p$ is $\debar$-closed;
in general, $R$ is $\nabla_f$-closed.

\begin{remark}
If $\J$ is an arbitrary ideal sheaf and $R$ is defined in the same way as above, then
\eqref{Rduality} holds, \cite{AW1}.  In case $\J$ is Cohen-Macaulay, one can express
this duality in a way that only involves the smooth form  $u$ in  $\Omega\setminus X$,
where $X$ is the zero set of $\J$, see \cite[Theorem~4.2]{AW1}. 
This result was recently proved algebraically in \cite{Lund} with no reference to
residue calculus and resolution of singularities. 
%%The statement  is  a generalization
%%of the classical Grothendieck duality for a zero-dimensional complete intersection ideal, and its
%%generalization to complete intersections due to Passare and Dickenstein-Sessa.
\end{remark}

\begin{remark}
In case $\J$ is generated by the single function $f$, then we  have the free  resolution
$0\to \Ok \stackrel{f}{\to} \Ok\to\Ok/(f)\to 0$;
thus $U$ is just the
principal value current $1/f$ and $R=\dbar(1/f)$.    
\end{remark}

Notice that \eqref{karvupplosn}  gives rise to the dual Hermitian complex
\begin{equation}\label{dualkomplex}
0\to \Ok(E_0^*) \stackrel{f^*_1}{\to}\cdots \to \Ok(E^*_{p-1})\stackrel{f^*_p}{\to}
\Ok(E_p^*)\stackrel{f^*_{p+1}}\longrightarrow\cdots .
\end{equation}
Since the sheaf $\Kers(\Ok(E_p^*)\stackrel{f^*_{p+1}}{\to}\Ok(E^*_{p+1}))$ is coherent,
there is a  (trivial) Hermitian vector bundle $F$ in  $\Omega$ and a holomorphic
morphism  $g\colon E_p\to F$ such that 
\begin{equation}\label{skutt}
\Ok(F^*)\stackrel{g^*}{\to}\Ok(E_p^*)\stackrel{f^*_{p+1}}{\to}\Ok(E^*_{p+1})
\end{equation}
is exact. 
Since $f_{p+1}$ has constant rank outside $X_{p+1}$, also $f_{p+1}^*$ has, and  it follows that
$g$ has as well. 
%%In fact, $g$ is just the first step in a free resolution of
%%the sheaf $\mathcal G=\Kers(\Ok(E_p^*)\stackrel{f^*_{p+1}}{\to}\Ok(E^*_{p+1}))$.
Outside $X_{p+1}$ we can thus define the mapping $\sigma_F\colon F\to E_p$ such that
$\sigma_F=0$ on $(\Image g)^\perp\subset F$, $\sigma_F g=Id$ on $(\Ker g)^{\perp}=(\Image f_{p+1})^\perp$
and $\Image \sigma_F$ is orthogonal to $\Ker g$. 
If $m=p$, then we can take $F=E_p$ and $g=Id$.

Let $d\zeta:=d\zeta_1\w\ldots\w d\zeta_N$. We also introduce the notation 
\begin{equation*}
E^r:=E_{p+r} |_{X}, \quad f^r:=f_{p+r}|_X
\end{equation*} 
so that $f^r$ becomes a holomorphic section of $\mbox{Hom}\, (E^{r}, E^{r-1})$.
Notice that for  $k\ge 1$,  $\alpha^k:=i^*\dbar\sigma_{p+k}$ are smooth in $X\setminus X^k$.

%%Via the structure form 
%%we can express the current $R$ as a principal value current {\em on} the variety $X$, cf.\ \eqref{DLrep0} and
%%\eqref{DLrep}.  Behavs denna extra kommentar????

\begin{proposition}\label{fundprop}
Let \eqref{karvupplosn} be a Hermitian free resolution of $\hol^{\Omega}/\J_X$ in $\Omega$ 
and let $R$ be the associated 
residue current.
Then there is a unique almost semi-meromorphic current
\begin{equation*}
\omega=\omega_0 + \omega_1 + \cdots + \omega_{n-1}
\end{equation*}
on $X$, where $\omega_{r}$ has bidegree $(n,r)$ and takes values in $E^r$, such that 
\begin{equation}\label{DLrep}
i_* \omega = R\wedge d\zeta.
\end{equation}
Moreover, %%The current $\omega$ is smooth on $X_{reg}$ and 
\begin{equation}\label{omega}
f^0 \omega_0 =0,\quad f^{r} \omega_{r} = \debar \omega_{r-1}, \,\, r\geq 1, \,\,\,\, \mbox{on} \,\,\, X,
\end{equation}
and 
\begin{equation}\label{omegauppsk}
|\omega| = \ordo (\delta^{-M})
\end{equation}
for some $M>0$, where $\delta$ is the distance to $X_{sing}$.

Assume that \eqref{skutt} is exact. 
The forms $\alpha^k$, $1\le k\le n-1$, defined and smooth outside $X^k$, and
$\sigma_F$, defined and smooth outside $X^1$,  extend to almost semimermorphic currents on  $X$.
%%%
%%
There is an $F$-valued section $\vartheta$ of $\Ba^X_n$ such that
\begin{equation}\label{omeganoll}
\omega_0=\sigma_F \vartheta.
\end{equation}
Moreover,
\begin{equation}\label{omega2}
\omega_{r}=\alpha_{r}\omega_{r-1}, \quad r\geq 1, \quad \mbox{on} \,\,\, X.
\end{equation}
%%%
\end{proposition}

We say that $\omega$ so obtained is a {\it structure form} on $X$. 
The products in \eqref{omeganoll} and \eqref{omega2} are well-defined by Proposition~\ref{multprop}.
Notice that if $X$ is Cohen-Macaulay and $m=p$, then $\omega_0$ is an $E^0$-valued section of $\Ba^X_n$. 

%%Observe that it follows in particular that $R\wedge \Phi =0$ if $\Phi$ is a smooth 
%%$(0,*)$-form in $\Omega$ whose pullback to $X_{reg}$ vanishes.

%%\marginpar{strukit att $R$ has SEP!}

\begin{proof}
Let $x$ be an arbitrary point on $X_{reg}$.
Since the ideal sheaf $\J_X$ is generated by the functions $f_1^j$ that constitute the 
map $f_1$, cf.\ \eqref{karvupplosn},
we can extract holomorphic functions $a_1\ldots,a_p$ from the $f_1^j$'s such
that $da_1\wedge \cdots \wedge da_p\neq 0$ at $x$. 
Possibly after a re-ordering of the variables $\zeta$ in the ambient space, we may assume that 
$\zeta=(\zeta',\zeta'')=(\zeta',\zeta''_1,\ldots,\zeta''_p)$ and that
$A:=\det (\partial a/\partial \zeta'')\neq 0$ at $x$. 
We also note that $d\zeta'\wedge da_1\wedge \cdots \wedge da_p =A d\zeta'\wedge d\zeta''=A d\zeta$ close to $x$.

Now, $\J_X$ is generated by $a=(a_1,\ldots,a_p)$ at $x$ and so the Koszul complex generated by the $a_j$ provides
a minimal resolution of $\hol^{\Omega}/\J_X$ there. The associated residue current $R^a=R^a_p$ is just the 
Coleff-Herrera product formed from the tuple $a$, cf., Section~\ref{exsektion}.
The original resolution \eqref{karvupplosn} contains the Koszul complex as a direct summand in a neighborhood of
$x$ and so it follows from Theorem~4.4 in \cite{AW1} that  
%%(blbala complete intersection is CM) 
%%
\begin{equation}\label{senap}
R_p=\alpha \, \debar \frac{1}{a_p}\wedge \cdots \wedge \debar \frac{1}{a_1},
\end{equation}
where $\alpha$ is a smooth section of $E_p$ close to $x$. By the Poincar\'e-Lelong formula thus 
\begin{eqnarray}\label{lok}
R_p\wedge d\zeta &=& \pm\alpha \, \debar \frac{1}{a_p}\wedge \cdots \wedge \debar \frac{1}{a_1} \wedge
da_1\wedge \cdots \wedge da_p \wedge \frac{d\zeta'}{A} \\
&=&
\pm (2\pi i)^p \alpha \frac{d\zeta'}{A}\wedge [X] \nonumber
\end{eqnarray}
close to $x$. If $\omega_0$ is  the pullback of $\pm (2\pi i)^p\alpha \, d\zeta'/A$ to $X_{reg}$, then 
the preceding equation means that
\begin{equation}\label{hast}
R_p\wedge d\zeta .\, \psi = \int_X \omega_0\wedge i^*\psi,
\end{equation}
where $\psi$ is a test form with support close to $x$. %% and $\phi$ is the pullback of $\Phi$ to $X_{reg}$.
Thus $\omega_0$ is determined by $R_p$ and so it extends to a global
$E_p$-valued $(n,0)$-form on $X_{reg}$,  still denoted $\omega_0$.
%%%
Since $\sigma_{p+1}u_p=0$ outside $X_{p+1}$, cf., \eqref{ukdef}, we find that $R_p$ and hence
$\omega_0$ takes values in $(\Image f_{p+1})^\perp\subset E_p$, cf., \eqref{rdef} and
\eqref{hast}.  Thus $\omega_0=\sigma_F g\omega_0=\sigma_F\vartheta $ where  $\vartheta:=g\omega_0$.
%%%
On $X_{reg}$ we have
$$
i_*\dbar \vartheta=-i_*g\dbar\omega_0=-g\dbar i_*\omega_0=
-g\dbar R_p\w d\zeta=-g f_{p+1} R_{p+1}\w d\zeta=0
$$
%%\dbar(g\omega_0)=-g\dbar\omega_0=-g f^1\omega_1=0,
%%$$
since $gf_{p+1}=0$. Thus $\dbar\vartheta=0$ and 
from  Example~\ref{ba} we conclude that $\vartheta$ is a section of $\Ba^X_n$.

Let $\mathfrak a_F$ be the Fitting ideal of $g$, restricted to $X$, i.e., the ideal (on $X$)
generated by the $r\times r$-minors of $g$, where $r$ is the generic
rank of $g$; notice that $g$ has rank $r$ on $X\setminus X^1$.
Let $\mathfrak a_k$ be the Fitting ideals of $f^k$, $k=1,\ldots,n-1$.  
By Hironaka's theorem there is a   smooth modification $\tau\colon\tilde X\to X$ 
such that all the ideals $\tau^*\mathfrak a_F, \tau^*\mathfrak a_1,\ldots, \tau^*\mathfrak a_{n-1}$ are
principal on $\tilde X$. This means that there are  holomorphic sections
 $s_F,s_1,\ldots, s_{n-1}$ of line bundles on $\tilde X$ that generate these ideals.
It follows from \cite[Lemma~2.1]{AW1} that $\tau^*\sigma_F=\beta_F/s_F$ and
$\tau^*\sigma^k=\beta_k/s_k$, $k\ge 1$, where $\beta_F$ and $\beta_k$ are smooth.
Hence, $\tau^*\alpha^k=\dbar\beta_k/s_k$.  We conclude that
$\sigma_F$ as well as $\alpha^k$ are almost semi-meromorphic on $X$. 

Let us now define $\omega_r$ inductively by \eqref{omega2}. We claim that
\begin{equation}\label{olvon}
i_*\omega_k=R_{p+k}\w d\zeta, \quad k\ge 0.
\end{equation}
If $k=0$ it is just \eqref{hast}. Assume \eqref{olvon} is proved for $k-1$. It follows from
\eqref{ukdef} and \eqref{rdef} that $R_{p+k}=\alpha_{p+k}R_{p+k-1}$ in 
$\Omega\setminus X_{p+k+1}$. In this set we thus have that
$$
i_*\omega_{k}=i_*\alpha^{k}\omega_{k-1}=\alpha_{p+k}i_*\omega_{k-1}=
\alpha_{p+k}R_{k-1}\w d\zeta.
$$
Let $\chi_\delta=\chi(|h|/\delta)$, where $h$ is a holomorphic tuple that cuts out $X_{p+k}$,
cf., \eqref{restrikdef2}.   Then 
 $i_*(i^*\chi_\delta\omega_k)=\chi_\delta R_k\w d\zeta$.
Now $\chi_\delta R_{p+k}\to R_{p+k}$
in view of \eqref{BEsats} and the dimension principle, and $i^*\chi_\delta\omega_k\to\omega_k$,
and hence 
%%%Since $i_*(i^*\chi_\delta\omega_k)=\chi_\delta R_k\w d\zeta$ it follows that 
\eqref{olvon} holds in $\Omega$.

%%By the dimension principle (applied on $X$ as well as on $\Omega$) and \eqref{BEsats},  
%%\eqref{olvon} holds in $\Omega$. 
%%To be precise, we do not know 
%%a~priori that the left hand side is in $\PM(\Omega)$ but  ????
The estimate 
\eqref{omegauppsk} follows  since it holds for $\varTheta$, being a tuple of 
meromorphic forms on $X$ that are holomorphic on  $X_{reg}$, and for
each of $i^*\sigma_F, \alpha^1,\ldots,\alpha^{n-1}$. 
Finally, \eqref{omega} follows since $(f-\debar)R=\nabla_fR=0$. %% cf.\ \eqref{Rdef}.
\end{proof}

Let $\varTheta$ be an $F$-valued meromorphic form in $\Omega$
such that $i^*\varTheta=\vartheta$. Notice that 
$$
\varTheta=\gamma_\varTheta\lrcorner d\zeta_1\w\ldots\w d\zeta_N 
$$
for a (unique) meromorphic section of  $F\otimes \Lambda^p T^{1,0}(\Omega)$. If 
$\gamma:= \sigma_F\gamma_\varTheta+\alpha^1\sigma_F\gamma_\varTheta+\alpha^2\alpha^1\sigma_F\gamma_\varTheta+\cdots$
and $\omega':=\gamma \lrcorner d\zeta$, thus
$
\omega=i^*\omega'.
$ 
%%
%%\omega'=\gamma\lrcorner d\zeta_1\w\ldots\w d\zeta_N 
%%$
%%is a form such that $i^*\omega'=\omega$.
%%Thus
Since  $[X]\w \gamma\lrcorner d\zeta=[X]\w\omega'=i_*\omega= R\w d\zeta$,
and   $[X]\w \gamma\lrcorner d\zeta=(-1)^p\gamma\lrcorner [X]\w d\zeta$
we  have 
\begin{equation}\label{staty}
R=(-1)^p\gamma\lrcorner [X], \quad i_*\omega=[X]\w\omega'=:[X]\w\omega.
\end{equation}
%%Since $[X]\w \gamma\lrcorner d\zeta=\pm \gamma\lrcorner [X]\w d\zeta$ we also find that
%%$R=\pm\gamma\lrcorner[X]$.

\begin{center}
---
\end{center}

We will now discuss
generalizations of  $(i)$ and $(ii)$ above. It is proved in \cite{Astrong} that 
if $\Phi$ is meromorphic in $\Omega$, then
$\phi:=i^*\Phi$ is in $\Ok^X$ if and only 
if  $\nabla_f (\Phi R)=0$ in $\Omega$. Combining with  Proposition~\ref{fundprop}
we get:

\smallskip
\noindent $(i)'$\ \  {\it If $\phi$ is a meromorphic
function on $X$, then $\phi$ is in $\Ok^X$ if and only if
$\nabla_f(\phi\omega)=0$ on $X$.}

\smallskip

Let $\varOmega^k$ denote the sheaf $\Ok(\Lambda^kT^*_{1,0}(\Omega))$.
Let $\xi d\zeta$ be a section of the sheaf 
$$
\Homs_\Ok(\Ok(E_p),\varOmega^N)\simeq
\Ok(E^*_p)\otimes_\Ok\varOmega^N
$$
such that $f^*_{p+1}\xi=0$. Then
$\dbar(\xi\cdot\omega_0)=-\xi\cdot\dbar\omega_0=-\xi\cdot f_{p+1}\omega_1=
f^*_{p+1}\xi\cdot\omega_1=0$, so that $\xi\cdot\omega_0$ is in $\Ba^X_n$.
The minus signs appear since  $f$ is an odd mapping with respect to the superstrucure.
Moreover, if $\xi=f^*_p\eta$ for $\eta\in\Ok(E^*_{p-1})$, then $\xi\cdot\omega_0=
f^*_p\eta\cdot\omega_0=-\eta\cdot f_p\omega_0=0$. 
We thus have a sheaf mapping
\begin{equation}\label{greta}
\Ho^p(\Homs(\Ok(E_\bullet),\varOmega^N))\to\Ba^X_n, \quad \xi d\zeta\mapsto \xi\cdot\omega_0.
\end{equation}

\begin{proposition}\label{takyta}
The mapping \eqref{greta} is an isomorphism, and it is independent of the specific choice
of resolution, hence establishing an isomorphism
$$
\Ext^p(\Ok^\Omega/\J_X,\varOmega^N)\simeq \Ba^X_n.
$$
\end{proposition}

This isomorphism is well-known, cf., \cite[Remark~5]{HePa}. 
%%and is the basis for the definition in \cite{Barlet}. 
Our contribution
is the realization  \eqref{greta}. Thus  $\Ba^X_n$ is coherent and we have:

\smallskip

\noindent $(ii)'$ \
{\it If $\xi_1, \ldots,\xi_\nu$ are  generators of 
$\Ho^p(\Homs(\Ok(E_\bullet^*)))$,  then 
$\eta_\ell:=\xi_\ell\cdot\omega_0, \ \ell=1,\ldots,\nu$,
generate the $\Ok^X$-module $\Ba^X_n$.}

\begin{proof}[Proof of Proposition~\ref{takyta}] 
If $h\in\Ba^X_n$, then $i_*h=h\w[X]$ is a so-called Coleff-Herrera current with respect to $X$
(taking values in the holomorphic vector bundle
$\Lambda^N T^*_{1,0}(\Omega)$)
that is annihilated by $\J_X$,  cf., \cite{Aext}.
Thus we have mappings
\begin{equation}\label{utter}
\Ho^p(\Homs(\Ok(E_\bullet),\varOmega^N)) %%\Ho^p(\Ok(E_\bullet^*)\otimes_\Ok\varOmega^N)
\to\Ba^X_n\to\Homs(\Ok^\Omega/\J_X,\CH_X)\otimes_\Ok\varOmega^N,
\end{equation}
defined by $\xi d\zeta\mapsto \xi\cdot\omega_0$ and $h\mapsto i_*h$.
The latter  mapping is certainly injective. The composed mapping is an isomorphism according
to \cite[Theorem~1.5]{Aext}.  It follows that both mappings are isomorphisms.
From \cite[Theorem~1.5]{Aext} we also know that the composed mapping  is independent
of the particular Hermitian resolution, and choice of $d\zeta$, and thus induces an isomorphism
$\Ext^p(\Ok^\Omega/\J_X,\varOmega^N)\simeq \Homs(\Ok^\Omega/\J_X,\CH_X)\otimes_\Ok\varOmega^N$.
Hence the proposition follows.
%%so we get indeed balbala.
\end{proof}

%% Combining these two proposition we conclude that a meromorphic form
%%$\phi$ on $X$ is indeed holomorphic if and only if $\dbar (\phi h)=0$ for all $h

%%%%%%%%%%%%%%%%%%%%
We conclude with a lemma that roughly speaking says that one can ``divide'' by $\omega$.

\begin{lemma}\label{omegalemma}
%%Let $\omega_0$ be the almost semi-meromorphic form on $X$ such that $i_*\omega_0 = R_p\wedge d\zeta$.
If $\phi$ is a smooth $(n,q)$-form on $X$, then there is a smooth $(0,q)$-form $\phi'$ on $X$ 
with values in $(E^0)^*$ such that 
$\phi=\omega_0\wedge \phi'$.
\end{lemma}

\begin{proof}
Let $\Phi$ be a smooth extension of $\phi$ to  $\Omega$.
Since $[X]$ is a Coleff-Herrera current (with values in $\Lambda^pT^*_{1,0}(\Omega)$), it follows from 
\cite[Theorem~1.5 and Example~1]{Aext} 
that locally there is a holomorphic $E_p^*$-valued $(p,0)$-form $a$ such that
$R_p\wedge a =[X]$. 

By a partition of unity we can find a global smooth $\tilde{a}$ such that 
$R_p\wedge \tilde{a} =[X]$ in $\Omega$. Since $\tilde{a}\wedge \Phi$ has bidegree $(N,q)$, there is an $E_p^*$-valued
smooth $(0,q)$-form $\Phi'$ in $\Omega$ such that $\tilde{a}\wedge \Phi=d\zeta\wedge \Phi'$. For every test form
$\Psi$ in $\Omega$ we now get
\begin{eqnarray*}
\int_X \phi\wedge i^*\Psi &=& [X]. (\Phi\wedge \Psi) =
R_p\wedge \tilde{a}. \,(\Phi\wedge \Psi) = R_p\wedge d\zeta. \,(\Phi'\wedge \Psi) \\
&=&
\int_X \omega_0\wedge \phi'\wedge i^*\Psi, 
\end{eqnarray*}   
where $\phi'=i^*\Phi'$.
Hence, $\phi=\omega_0\wedge \phi'$ on $X$.
\end{proof}

An algebraic counterpart of the factorization $R_p\w a=[X]$ appeared
in \cite{LJ}  in case $X$ is Cohen-Macaulay; then  one can take
$a=df_1 df_2\cdots df_p$. %%%, provided that $m=p$ in \eqref{karvupplosn}.

\section{The  strong $\debar$-operator on $X$}\label{starksektion}

Let $\omega$ be a structure form on  $X$,
and let $\chi_\delta:=\chi(|h|/\delta)$, where $\chi$ is a smooth approximand
of the characteristic  function of $[1,\infty)$, and  $h$ is a 
holomorphic tuple such that $X_{sing}=\{h=0\}$.
Notice that if $\alpha\in\W(X)$, then
\begin{equation}\label{pellepenna}
\1_{X_{sing}}\nabla_f\alpha=0 \iff \1_{X_{sing}}\dbar\alpha=0 \iff
\dbar\chi_\delta\w\alpha\to 0,\   \delta\to 0.
\end{equation}
In fact,  since $\1_{X_{sing}}\alpha=0$ and $f$ is smooth we have that 
$\1_{X_{sing}}f\alpha=0$; hence the first equivalence follows. For the second one, consider the equality
$$
\dbar(\chi_\delta\alpha)=\chi_\delta\dbar\alpha+\dbar\chi_\delta\w\alpha.
$$
Since $\chi_\delta\alpha\to\alpha$ it follows that 
$\1_{X_{sing}}\dbar\alpha=\lim(1-\chi_\delta)\dbar\alpha=0$ if and only if $\dbar\chi_\delta\w\alpha\to 0$.

\begin{lemma}\label{urdjur}
Assume that $\mu\in\W(X)$.

\begin{itemize}
\item[(i)] If there is $\tau\in\W(X)$ such that 
\begin{equation}\label{urdum}
-\nabla_f (\mu \wedge \omega) = \tau\wedge \omega,
\end{equation}
then $\dbar\mu=\tau$ and
\begin{equation}\label{urbota}
\dbar\chi_\delta\w\mu\w\omega\to 0,\ \delta\to 0.
\end{equation}

\item[(ii)] If $\dbar\mu\in\W(X_{reg})$ and \eqref{urbota} holds, then there is $\tau\in\W(X)$ such that
\eqref{urdum} holds.
\end{itemize}
\end{lemma}

From Proposition \ref{multprop} we know that $\mu\wedge \omega$ is a well-defined current in $\W(X)$.

\begin{proof}
Assume that \eqref{urdum} holds. Then  $-\1_{X_{sing}}\nabla_f (\mu \wedge \omega)=
\1_{X_{sing}} \tau\wedge \omega=0$, since $\tau\w\omega$ is in $\W(X)$. Thus \eqref{urbota}
holds, in view of \eqref{pellepenna}. Since $\omega$ is smooth on $X_{reg}$
 and $\nabla_f\omega=0$, \eqref{urdum} implies that $\dbar\mu\w\omega_0=\tau\w\omega_0$
on $X_{reg}$. It follows from  Lemma~\ref{omegalemma} that
$\dbar\mu=\tau$ on $X_{reg}$.  Moreover, from \eqref{urbota} and  Lemma~\ref{omegalemma}
we find that $\dbar\chi_\delta\w\mu\to 0$, so that, cf., \eqref{pellepenna},
$\1_{X_{sing}}\dbar \mu=0$. It follows that $\dbar\mu=\1_{Xreg}\dbar\mu=\tau$. Thus
(i) is proved.

\smallskip

If  $\dbar\mu$ has the SEP on $X_{reg}$, then $\tau:=\1_{X_{reg}}\dbar\mu$ has the SEP on $X$
and hence is in $\W(X)$.  
%%As in the proof of (i),  \eqref{urbota} implies that
%%$\1_{X_sing}\dbar\mu=0$. Thus $\tau=\dbar\mu$.  
Since $\omega$ is smooth on $X_{reg}$,
$-\nabla_f(\mu\w\omega)=\tau\w\omega$ there. In view of \eqref{urbota} and \eqref{pellepenna},
$\1_{X_{sing}}\nabla_f(\mu\w\omega)=0$,  and since $\nabla_f(\mu\w\omega)$ has the SEP on $X_{reg}$
it follows that it has the SEP on $X$, i.e., is in $\W(X)$.
Since $\omega$ is smooth on $X_{reg}$, \eqref{urdum} holds on $X_{reg}$. Since both sides have the
SEP, the equality must hold on $X$.
\end{proof}

%%Since $\omega$ is smooth on $X_{reg}$ and  $\nabla_f$-closed,
%%it follows that $\dbar\mu\w\omega=\tau\w\omega$ 
%%and in particular  $\dbar\mu\w\omega_0=\tau\w\omega_0$ 
%%there. From 
%% Lemma~\ref{omegalemma} (or Proposition~\ref{takyta}) we conclude that
%%$\dbar\mu=\tau$ on  $X_{reg}$. 
%%Now take
%%$\chi_\delta:=\chi(|h|/\delta)$ where  $h$ is a tuple whose common zero set is equal to $X_{sing}$,
%%and $\chi$ is a smooth approximand of the characteristic function for the interval $[1,\infty)$.
%%Since $\chi$ vanishes in a \nbh of $X_{sing}$ we have 
%%$$
%%-\nabla_f(\chi_\delta\mu\w\omega)=\chi_\delta\tau\w\omega+\dbar\chi_\delta\w\mu\omega.
%%$$
%%By the SEP, $\chi_\delta\mu\w\omega\to \mu\w\omega$ and $\chi_\delta\tau\omega\to\tau\w\omega$
%%when $\delta\to 0$.
%%We conclude that $\dbar\chi_\delta\w\mu\omega\to 0$. It  follows from
%%Lemma~\ref{omegalemma} that $\dbar\chi_\delta\w\mu\to 0$.  From the equality
%%$\dbar(\chi_\delta\mu)=\chi_\delta\tau+\dbar\chi_\delta\w\mu$ we conclude that $\dbar\mu=\tau$ on $X$.
%%If $\mu\in\W_0(X)$ and $\tau=0$, then $\mu$ is holomorphic on $X_{reg}$ with a meromorphic extension
%%to  $X$, cf., Example~\ref{ba}, and by the SEP,  $\mu$ must coincide with this meromorphic function.
%%From Proposition~\ref{golvyta} we conclude that  $\mu\in\Ok^X$. 
%%\end{proof}

Let $x$ be a point in an arbitrary complex space $X$. By choosing local embeddings 
$X\hookrightarrow \Omega \subset \C^N$
at $x$ and Hermitian free resolutions of $\hol^{\Omega}/\J_X$ (and choice of
coordinates on $\Omega$, cf., \eqref{DLrep}) 
we get the collection $\mathfrak{S}_x$ of all  structure forms $\omega$ at $x$.

%%Notice that $\omega$ also depends on the choice of 
%%holomorphic coordinates in the ambient space since $i_*\omega=R\wedge d\zeta$,
%%but only up to an invertible holomorphic function.

Given $\mu,\tau\in \W^X_{0,*,x}$ we say that $\debar_X \mu = \tau$ at $x$ if
\eqref{urdum} holds at $x$ 
for all
$\omega\in \mathfrak{S}_x$. 
It follows from  Lemma~\ref{urdjur} that $\dbar_X\mu=\tau$ if and only if $\dbar\mu=\tau$
and the ``boundary condition'' \eqref{urbota} holds for every $\omega\in \mathfrak{S}_x$.

\begin{definition}[The sheaves $\mbox{Dom}_q\, \debar_X$]\label{domdef}
We say that a $(0,q)$-current $\mu$ is a section of $\mbox{Dom}_q\,\debar_X$ in  the open set $\mathcal{U}\subset X$
if $\mu\in \W_{0,q}(\mathcal{U})$ and there is $\tau\in\W_{0,q+1}(\mathcal U)$ such  that $\dbar_X\mu=\tau$
 in $\mathcal U$, i.e., $\dbar_X\mu=\tau$ at each point $x\in\mathcal U$.
\end{definition}

If $\mu\in\W(X)$ is smooth on $X_{reg}$, then it follows from Lemma~\ref{urdjur} that
$\mu\in\mbox{Dom}_q\, \debar_X$ if and only if \eqref{urbota} holds at each $x\in X$ for each
 $\omega\in \mathfrak{S}_x$.  If $\mu$ is smooth on $X$, then 
$\dbar(\mu\w\omega)$ has the SEP, and so  \eqref{urbota} holds for each $\omega$. Thus 
$\E_{0,q}^X$ is a subsheaf of $\mbox{Dom}_q\, \debar_X$.

\begin{proposition}\label{Bprop}
The sheaves $\mbox{Dom}_*\,\debar_X$ are $\E_{0,*}^X$-modules and 
\begin{equation}\label{Bkomplex}
0 \to \hol^X \hookrightarrow \mbox{Dom}_0\,\debar_X \stackrel{\debar}{\longrightarrow} \mbox{Dom}_1\,\debar_X 
\stackrel{\debar}{\longrightarrow} \cdots
\end{equation}
is a complex. Moreover, the kernel of $\debar$ in $\mbox{Dom}_0\,\debar_X$ is $\hol^X$.
\end{proposition}

When $\dim X=1$ the complex  \eqref{Bkomplex} is exact, i.e., a fine resolution of $\hol^X$,
see Section \ref{kurvsektion} below. We do not know
whether this is true if $\dim X >1$.

\begin{proof}
Assume that $\mu$ is in $\text{Dom}\,\debar_X$ and that $\omega \in \mathfrak{S}_x$.
In view of Lemma~\ref{urdjur},
$\mu$ and $\debar \mu$ are both in $\W^X$ and \eqref{urbota} holds.
Since  $\omega$ is smooth on $X_{reg}$ and $\nabla_f\omega=0$, 
$\nabla_f(\dbar\chi_\delta\w\mu\w\omega)=-\dbar\chi_\delta\w\dbar\mu\w\omega$.
Therefore \eqref{urbota}, with $\mu$ replaced by $\dbar\mu$, holds as well
and it follows from Lemma~\ref{urdjur} that $\debar\mu \in \text{Dom}\, \debar_X$. 
Moreover, if $\xi$ is smooth it is 
clear that \eqref{urbota}
holds with $\mu$ replaced by $\xi\w\mu$. Since $\dbar(\dbar\mu)=0$ and 
$\dbar(\xi\w\mu)$ is in $\W^X$ we conclude that $\xi\wedge \mu \in \text{Dom}\, \debar_X$.

Now assume $\mu\in\W^X_{0,0}$ and 
\eqref{urdum} holds with $\tau=0$. Then $\dbar\mu=0$ by Lemma~\ref{urdjur} and
hence $\mu$ is holomorphic
on $X_{reg}$, and has a meromorphic extension to $X$, cf., Example~\ref{ba}.  Thus
$\mu\in\Ok^X$ in view of $(i)'$ above.%%Proposition~\ref{golvyta}.
\end{proof}

If \eqref{urdum} holds at $x$ for a given $\omega\in\mathfrak{S}_x$, 
then in particular, $\dbar(\mu\w\omega_0)\pm\mu f_{p+1}\omega_1=
\tau\w\omega_0$. Applying various $\xi\in\Ok(E_p^*)$ with $f^*_{p+1}\xi=0$ 
to this equality we conclude, 
by  Proposition~\ref{takyta},  that
\begin{equation}\label{hvil}
\dbar(\mu\w \theta)=\tau\w \theta,\quad  \theta\in\Ba^X_x.
\end{equation}

If $X$ is Cohen-Macaulay, and \eqref{urdum} holds for one $\omega$, then it holds
for all $\omega\in \mathfrak{S}_x$. In fact we have:

\begin{proposition} If $X$ is Cohen-Macaulay, then $\mu\in\W(X)$ is in 
$\mbox{Dom}\,\debar_X$ and $\dbar_X \mu=\tau$   if and only if
%%there is $\tau\in\W(X)$ such that 
(locally)  \eqref{hvil} holds.
\end{proposition}

\begin{proof}
It follows from Proposition~\ref{fundprop} that if  $X$ is Cohen-Macaulay at $x\in X$,
and thus $X^1=\emptyset$,  any 
$\omega\in\mathfrak{S}_x$ has the form 
$a\vartheta$ where $\vartheta$ is (a vector-valued) section of $\Ba^X$ and $a$ is smooth.
If $\tau\in\W(X)$ and \eqref{hvil} holds, then
$$
\dbar(\mu\w\omega)=\pm \dbar(a\mu\w \vartheta)=\pm \dbar a\w\mu\w \vartheta 
\mp a\dbar(\mu\w \vartheta)=
\pm \dbar a\w\mu\w \vartheta \mp a \tau\w \vartheta.
$$
It  follows that $\1_{X_{sing}}\dbar(\mu\w\omega)=0$ and hence $\mu$ is in 
$\mbox{Dom}\,\debar_X$. 
\end{proof}

Notice that $\W^X_{0,n} = \mbox{Dom}_n\, \debar_X$. Assume now that 
\begin{equation}\label{Serre}
\mbox{codim}\, X^r \geq r+\ell, \quad r\geq 0.
\end{equation}
We claim that if $q\leq \ell-2$, 
 $\mu\in\W^X_{0,q}$ and $\dbar\mu\in\W^X_{0,q+1}$, 
then $\mu\in \mbox{Dom}_q\,\debar_X$.
To see this,  we have to verify that $\1_{X_{sing}}\dbar(\mu\w\omega_k)=0$
for each $k\ge 0$. For $k=0$ it follows directly by the dimension principle
since $\1_{X_{sing}}\dbar(\mu\w\omega_0)$ has bidegree (at most) $(n,\ell-1)$
and support on $X^0$ that has codimension $\ell$. Now, $\omega_1=\alpha^1\omega_0$
and $\alpha^1$ is smooth outside $X^1$, so $\1_{X_{sing}}\dbar(\mu\w\omega_1)=
\pm\alpha^1\1_{X_{sing}}\dbar(\mu\w\omega_0)=0$ outside $X^1$. Thus
$\1_{X_{sing}}\dbar(\mu\w\omega_1)$ has support on $X^1$ and hence must  vanish by 
\eqref{Serre} and the dimension principle. The claim follows in this way by induction.
It follows in particular that $X$ is normal if \eqref{Serre} holds for $\ell=2$.
One can verify that \eqref{Serre} with  $\ell=2$  is a way to formulate 
Serre's conditions  $R1$ and $S2$ for normality.

\section{Koppelman formulas on $X\subset\Omega$ (the embedded context)}\label{KoppelsektionI}

We first recall the construction of integral formulas in \cite{Aint} on an open set $\Omega\subset \C^N$.
Let $\eta=(\eta_1,\ldots,\eta_N)$ be a holomorphic tuple in $\Omega_{\zeta}\times \Omega_z$ that generates the ideal 
associated with the diagonal $\Delta \subset \Omega_{\zeta}\times \Omega_z$. For instance one can take
$\eta=\zeta-z$. 
Following the last section in \cite{Aint} we 
consider forms in $\Omega_{\zeta}\times \Omega_z$ with values
in the exterior algebra $\Lambda_{\eta}$ spanned by $T^*_{0,1}(\Omega\times \Omega)$ and
the $(1,0)$-forms $d\eta_1,\ldots,d\eta_n$.
On such forms interior multiplication $\delta_\eta$ with
\begin{equation*}
\eta=2\pi i \sum_1^N\eta_j\frac{\partial}{\partial \eta_j}
\end{equation*}
is defined.  Let $\nabla_{\eta}=\delta_{\eta}-\debar$.
A smooth section  $g=g_{0}+\cdots +g_{N}$ of $\Lambda_{\eta}$, defined for $z\in \Omega' \Subset \Omega$
and $\zeta \in \Omega$, 
such that $\nabla_{\eta}g=0$ and $g_{0}|_{\Delta}=1$,  lower indices denote degree in $d\eta$,
will be called a {\em weight with respect to $z\in\Omega'$}.
Notice that if $g$ and $g'$ are weights, then $g\wedge g'$ is again a weight. We will use one weight that has 
compact support in $\Omega$ %%(with respect to $z\in \Omega'\subset\subset\Omega$) 
and one weight which gives a division-interpolation type formula (for $z\in\Omega'$)
for the ideal sheaf $\J_X$ associated with a subvariety $X\hookrightarrow \Omega$.
We first discuss weights with compact support.

\begin{example}[Weights with compact support]\label{vikt}
If $\Omega$ is pseudoconvex and $K$ is a holomorphically convex compact subset, then one can find a weight with 
respect to $z$ in some neighborhood $\Omega'\Subset\Omega$ of $K$, 
depending holomorphically on $z\in \Omega'$, that has compact support 
in  $\Omega$, see, e.g., Example 2 in \cite{Aint2}.
Here is an explicit choice when $\Omega$ is a neighborhood of the closed unit ball $\overline{\mathbb{B}}$,
$K=\overline{\B}$, and $\eta=\zeta-z$: 
Let $\sigma = \bar{\zeta}\cdot d\eta/(2\pi i(|\zeta|^2-\bar{\zeta}\cdot z))$. Then $\delta_{\eta}\sigma =1$
for $\zeta\neq z$ and 
\begin{equation*}
\sigma \wedge (\debar \sigma)^{k-1}=
\frac{1}{(2\pi i)^k}\frac{(\bar{\zeta}\cdot d\eta)\wedge 
(d\bar{\zeta}\cdot d\eta)^{k-1}}{(|\zeta|^2-\bar{\zeta}\cdot z)^k}.
\end{equation*}
If  $\chi=\chi(\zeta)$ is a cutoff function that is $1$ in a slightly larger ball  $\Omega'$, 
then %%%we can take (as a weight with respect to $z\in\Omega'$ and with compact support in $\Omega$)
\begin{equation*}
g=\chi-\debar \chi \wedge \frac{\sigma}{\nabla_{\eta}\sigma}=
\chi-\debar \chi\wedge \Big(\sum_{k=1}^N\sigma \wedge (\debar \sigma)^{k-1}\Big).
\end{equation*}
is a weight with respect to $z\in\Omega'$  with compact support in $\Omega$.
\hfill $\Box$
\end{example}

Let $s$ be a smooth $(1,0)$-form in $\Lambda_{\eta}$ such that $|s|\lesssim |\eta|$ and 
$|\eta|^2\lesssim |\delta_{\eta} s|$; such an $s$ is called admissible. Then 
$B:=s/\nabla_{\eta}s=\sum_k s\wedge (\debar s)^{k-1}$ satisfies $\nabla_{\eta} B=1-[\Delta]$, where
$[\Delta]$ is the $(N,N)$-current of integration over $\Delta$. %%\subset \Omega \times \Omega$.
If $\eta=\zeta-z$,  then $s=\partial |\eta|^2$ will do and we refer to the resulting $B$ as the Bochner-Martinelli
form. If $g$ is any weight, we have $\nabla_{\eta} (g\wedge B) = g - [\Delta]$, and identifying terms of 
bidegree $(N,N-1)$ we see that
\begin{equation}\label{Koppel}
\debar (g\wedge B)_N=[\Delta]-g_N,
\end{equation}
which is equivalent to a weighted Koppelman formula in $\Omega$.

\begin{center}
---
\end{center}

\noindent We now turn our attention to construction of weights
for division-interpolation with respect to the ideal $\J_X$.
For the rest of this section we assume that $\Omega\subset \C^N$ is pseudoconvex and that 
$X\hookrightarrow \Omega$ is a subvariety. 
Let us fix global holomorphic frames for the bundles $E_k$ in
\eqref{VBkomplex} over $\Omega$. Then $E_k\simeq \C^{\mbox{rank}\,E_k}\times \Omega$, and the morphisms 
$f_k$ are just matrices of holomorphic functions. One can find, see \cite{Aint2} for explicit choices,
$(k-\ell,0)$-form-valued Hefer morphisms, i.e., matrices $H_k^{\ell}\colon E_k\to E_{\ell}$,
depending holomorphically on $z$ and $\zeta$,  such that $H_k^k =I_{E_k}$ and 
\begin{equation*}
\delta_{\eta}H_k^{\ell}=H_{k-1}^{\ell}f_k-f_{\ell+1}(z) H_k^{\ell+1}, \quad k>\ell,
\end{equation*}
where $I_{E_k}$ is the identity operator on $E_k$ and $f$ stands for $f(\zeta)$. 
For $\real\, \lambda\gg 0$ we put $U^{\lambda}=|F|^{2\lambda}u$, see Section~\ref{Rsektion} for the notation, and 
\begin{equation*}
R^{\lambda}=\sum_{k=0}^N R_k^{\lambda}= 1-\nabla_f U^{\lambda}=
1-|F|^{2\lambda} + \debar |F|^{2\lambda}\wedge u.
\end{equation*}
Then 
\begin{eqnarray*}
g^{\lambda}&:=& 1-\nabla_{\eta}\sum_{k=1}^N H_k^0 U_k^{\lambda}=
\sum_{k=0}^N H_k^0 R_k^{\lambda} + f_1(z)\sum_{k=1}^N H_k^1 U_k^{\lambda} \\
&=&
HR^{\lambda} + f_1(z) HU^{\lambda}
\end{eqnarray*}
is a weight that is as smooth as we want if $\real\, \lambda$ is large enough.
Let $g$ be any smooth weight with respect to $\Omega'\Subset \Omega$ (but not necessarily holomorphic
in $z$) with compact support in $\Omega_{\zeta}$. Then  \eqref{Koppel} holds with $g$ replaced by
$g^{\lambda}\wedge g$. Since $R(z)$ is $\nabla_{f(z)}$-closed we thus get
\begin{eqnarray*}
-\nabla_{f(z)} \big(R(z)\wedge dz\wedge (g^{\lambda}\wedge g\wedge B)_N\big) &=&
R(z)\wedge dz\wedge [\Delta] -\\
&-& R(z)\wedge dz\wedge (g^{\lambda}\wedge g)_N.
\end{eqnarray*}
Notice that the products of  currents are well-defined; they are just tensor products 
since $z$ and $\eta$ are independent variables in $\Omega\times \Omega$. 
Moreover, since $R(z) f_1(z)=0$ we have
\begin{eqnarray}\label{ASKoppel}
-\nabla_{f(z)} \big(R(z)\wedge dz\wedge (HR^{\lambda}\wedge g\wedge B)_N\big) &=&
R(z)\wedge dz\wedge [\Delta] - \nonumber \\
&-& R(z)\wedge dz\wedge (HR^{\lambda}\wedge g)_N.
\end{eqnarray}
It follows from \eqref{DLrep} that (recall that $\Delta \subset \Omega\times \Omega$ is the diagonal) 
\begin{equation}\label{ahh}
R(z)\wedge dz\w [\Delta]= \iota_* \omega,
\end{equation}
where $\iota \colon \Delta^X \hookrightarrow \Omega\times \Omega$ is the inclusion of the 
diagonal $\Delta^X \subset X\times X \subset \Omega\times \Omega$. We notice that 
the analytic continuation to $\lambda=0$ of the last term on the right hand side of \eqref{ASKoppel} exists
and yields the well-defined current $R(z)\w dz\wedge (HR\wedge g)_N$ in $\Omega_{\zeta}\times \Omega'_z$.
The existence of the analytic continuation to $\lambda=0$ of the left hand side of \eqref{ASKoppel} 
follows from Proposition~2.1 in \cite{AW2} since $R(z)\wedge B$ is pseudomoromorphic in $\Omega\times \Omega$.
Our Koppelman formulas will follow by letting $\lambda=0$ in \eqref{ASKoppel}.

%% this will give \eqref{prelim} and Lemma~\ref{ko}.

\smallskip

%Let $S=S(\zeta,z)$ be a smooth (or sufficiently differentiable) form with values in $\Lambda_{\eta}$. 
To begin with, let us consider \eqref{ASKoppel} for $\lambda=0$ in
$(\Omega\setminus X_{sing})\times (\Omega'\setminus X_{sing})$.
In this set we have, by \eqref{DLrep} and \eqref{staty}, that 
\begin{equation}\label{huga}
R(z)\wedge dz \wedge (HR\wedge g)_N =
\pm \omega(z)\wedge [X_z]\wedge \big( H(\gamma(\zeta) \lrcorner [X_{\zeta}])\wedge g \big)_N \quad \quad
\end{equation}
\begin{equation*}
\quad \quad = \pm \omega(z)\wedge [X_z]\wedge [X_{\zeta}] \wedge \gamma(\zeta)\lrcorner (H\wedge g)_N 
= \omega(z)\wedge [X_z \times X_{\zeta}] \wedge p(\zeta,z),
\end{equation*} 
where 
\begin{equation}\label{p}
p(\zeta,z):=\pm(\gamma(\zeta)\lrcorner (H\wedge g)_N)_{(n)}
\end{equation} 
is the term of 
$\pm\gamma(\zeta)\lrcorner (H\wedge g)_N$ of degree $n$ in $d\zeta$; this is the only term of 
$\pm\gamma(\zeta)\lrcorner (H\wedge g)_N$ that can contribute in \eqref{huga}
since $\omega(z)\wedge [X_z]$ has full degree in the $dz_j$.
Notice that $p(\zeta,z)$ is almost semi-meromorphic on $X\times X'$ 
($X'=X\cap\Omega'$)  and smooth on $X_{reg}\times \Omega'_z$;
if $g$ is holomorphic in $z$ then $z\mapsto p(\zeta,z)$ is holomorphic in $\Omega'$.

\begin{lemma}\label{penna}
In $(\Omega_{\zeta}\setminus X_{sing})\times (\Omega'_z\setminus X_{sing})$ we have
\begin{equation}\label{tu}
R(z)\wedge dz \wedge (HR^{\lambda}\wedge g\wedge B)_N|_{\lambda=0}=
R(z)\wedge dz \wedge (HR\wedge g\wedge |\eta|^{2\lambda}B)_N|_{\lambda=0}.
\end{equation}
\end{lemma}

\begin{proof}
Recall from Section~\ref{Rsektion} that in $\Omega\setminus X_{sing}$, $R$ is a smooth form times
$R_p$. Notice that
\begin{equation*}
T_{jk}:=R_p(z)\wedge R^{\lambda}_k\wedge B_{j-k}|_{\lambda=0} -
R_p(z)\wedge R_k \wedge |\eta|^{2\lambda} B_{j-k}|_{\lambda=0}, \,\,\,j\leq N.
\end{equation*}
is a pseudomeromorphic current in $\Omega\times \Omega$
of bidegree $(j-k,p+k+j-k-1)=(j-k,p+j-1)$ that clearly vanishes outside $X_z$. It also vanishes
outside $\Delta$ since $B$
is smooth there. Thus $T_{jk}$ has support contained in $\Delta^X\simeq X$, which has codimension 
$2N-n=p+N$ in $\Omega\times\Omega$. Since $p+N>p+j-1$ for $j\leq N$, it follows from the dimension principle
that $T_{jk}$ must vanish; in particular, $R_p(z)\wedge R^{\lambda}_k\wedge B_{j-k}|_{\lambda=0}=0$ 
for $k<p$ since $R_k=0$ for $k<p$. We conclude that \eqref{tu} holds in 
$(\Omega\setminus X_{sing})\times (\Omega'\setminus X_{sing})$ since $T_{jk}=0$ there.
\end{proof}

Notice that the right hand side of \eqref{tu} only involves $B_j$ with $j\leq n$ since all terms in
$HR$ have degree at least $p$ in $d\eta$. 
If $\mathfrak{Re} \, \lambda \gg 0$ we may replace $g$ by $g\wedge |\eta|^{2\lambda}B$ in \eqref{huga} 
and combining with Lemma~\ref{penna} we get
\begin{equation}\label{ipod}
R(z)\wedge dz \wedge (HR^{\lambda}\wedge g\wedge B)_N\big|_{\lambda=0}= 
R(z)\wedge dz \wedge (HR\wedge g\wedge |\eta|^{2\lambda} B)_N\big|_{\lambda=0}
\end{equation}
\begin{eqnarray*}
\,\,\,\,\, &=&\omega(z)\wedge [X\times X] \wedge \big(\gamma(\zeta) \lrcorner 
(H\wedge g\wedge |\eta|^{2\lambda}B)\big)_N \big|_{\lambda=0} \\
\,\,\,\,\, &=& \omega(z)\wedge [X\times X] \wedge 
\big(\gamma(\zeta)\lrcorner \sum_{j=1}^n(H\wedge g)_{N-j}\wedge |\eta|^{2\lambda}B_j\big)\big|_{\lambda=0}
\end{eqnarray*}
in $(\Omega\setminus X_{sing})\times (\Omega'\setminus X_{sing})$. 
Since $B_j=\ordo(|\eta|^{-2j+1})$, we see that $B_j$ is locally integrable on $X_{reg}\times X_{reg}$ for $j\leq n$.
It is thus innocuous to put $\lambda=0$ in the right hand side of \eqref{ipod} as long as we restrict our attention 
to $X_{reg}\times X'_{reg}$. Notice that the integral kernel
\begin{equation}\label{k}
k(\zeta,z) := \pm \big(\gamma(\zeta)\lrcorner \sum_{j=1}^n(H\wedge g)_{N-j}\wedge B_j\big)_{(n)}
\end{equation}
is almost semi-meromorphic on $X\times X'$ and locally integrable on $X_{reg}\times X'_{reg}$.

In view of \eqref{ASKoppel}, \eqref{ahh}, \eqref{huga}, \eqref{ipod}, and \eqref{k} we have that
\begin{equation}\label{prelim}
-\nabla_{f(z)} \big(\omega(z)\wedge k(\zeta,z)\big) = 
\omega \wedge [\Delta^X] - \omega(z)\wedge p(\zeta,z)
\end{equation}
in the current sense on $X_{reg}\times X'_{reg}$. Combined with Lemma~\ref{omegalemma} this gives
\begin{lemma}\label{ko}
With $k(\zeta,z)$ and $p(\zeta,z)$ defined by \eqref{k} and \eqref{p} respectively, we have
\begin{equation*}
\debar k(\zeta,z)=[\Delta^X] - p(\zeta,z)
\end{equation*}
in the current sense on $X_{reg}\times X'_{reg}$.
\end{lemma}

We can write our integral kernels $p(\zeta,z)$ and $k(\zeta,z)$ in terms of the 
structure form $\omega$ as follows:
Let $F$ be a trivial vector bundle over $\Omega\times \Omega$ with basis elements
$\epsilon_1,\ldots,\epsilon_N$. Now replace each occurrence of $d\eta_j$ in $H$ and $g$ by 
$\epsilon_j$ and let $\hat{H}$ and $\hat{g}$ be the forms so obtained. Then
\begin{equation*}
(H\wedge g)_N=\epsilon^*_N\wedge \cdots \wedge \epsilon^*_1\lrcorner \big(d\eta_1\wedge \cdots\wedge d\eta_N
\wedge (\hat{H}\wedge \hat{g})_N\big),
\end{equation*} 
where $\{\epsilon^*_j\}$ is the dual basis and the lower index $N$ on the right hand side means 
the term with degree $N$ in the $\epsilon_j$. If $C=C(\zeta,z)$ is the invertible holomorphic function 
defined by $d\eta =C d\zeta+ \cdots$, we thus have, cf., \eqref{staty},
\begin{equation}\label{palt}
p(\zeta,z)
\pm C \epsilon^*_N\wedge \cdots \wedge \epsilon^*_1 \lrcorner (\hat{H}\wedge \hat{g})_N\wedge \omega(\zeta).
\end{equation}
Similarly, we get that
\begin{equation}\label{kalt}
k(\zeta,z) =
\pm C \epsilon^*_N\wedge \cdots \wedge \epsilon^*_1\lrcorner 
\sum_{j=1}^n(\hat{H}\wedge \hat{g})_{N-j}\wedge \hat{B}_j \wedge \omega(\zeta).
\end{equation}

\section{Koppelman formulas on $X$ (in the intrinsic context)}\label{KoppelsektionII}
Let $X$ be a reduced complex space of pure dimension $n$. Locally, $X$ can be embedded as a subvariety 
of a pseudoconvex domain in some $\C^N$, so let us, for notational convenience, assume that $X$ can be embedded,
$X\hookrightarrow \Omega$, in a pseudoconvex domain $\Omega\subset \C^N$.
Then, following the previous section, for any $\Omega'\Subset \Omega$ we can construct integral
kernels $k(\zeta,z)$ and $p(\zeta,z)$ which are almost semi-meromorphic on $X\times X'$, where $X'=X\cap \Omega'$,
such that \eqref{prelim} and Lemma~\ref{ko} hold. Moreover, $k(\zeta,z)$ and $p(\zeta,z)$ are
locally integrable on $X_{reg}\times X'_{reg}$ and smooth on $X_{reg}\times \Omega'$ respectively.
 
Now assume that $\mu(\zeta)\in \W_{0,q}(X)$. Since $k(\zeta,z)$ and $p(\zeta,z)$ are almost semi-meromorphic,
the products $k(\zeta,z)\wedge \mu(\zeta)$ and $p(\zeta,z)\wedge \mu(\zeta)$ are well-defined currents 
in $\W(X_{\zeta}\times X'_z)$ in view of Proposition~\ref{multprop}. 
Let $\pi\colon X_{\zeta}\times X_z\to X_z$ be the projection
and put $\K \mu (z) = \pi_* (k(\zeta,z)\wedge \mu(\zeta))$ and
$\Proj \mu (z) = \pi_*(p(\zeta,z)\wedge \mu(\zeta))$. Since $k$ and $p$ have compact support in $\zeta\in \Omega$,
$\K \mu$ and $\Proj \mu$ are well-defined currents in $\W(X'_z)$, and in fact, $\Proj \mu (z)$ is a 
smooth function in $\Omega'$ since $p(\zeta,z)$ is smooth in $z\in\Omega'$; if we choose the weight $g$ to be holomorphic 
in $z$, then $\Proj \mu (z)$ is holomorphic in $\Omega'$.
It is of course natural to write
\begin{equation}\label{Kop}
\K \mu (z) = \int_{X_{\zeta}} k(\zeta,z)\wedge \mu(\zeta), \quad
\Proj \mu (z) = \int_{X_{\zeta}} p(\zeta,z)\wedge \mu(\zeta).
\end{equation}

\begin{lemma}\label{smooth}
Let $\mu \in \W_{0,q}(X)$.
\begin{itemize}
\item[(i)] If $\mu$ is smooth in a neighborhood of a given point $x\in X'_{reg}$, then $\K \mu (z)$ is smooth
in a neighborhood of $x$.

\item[(ii)] If $\mu$ vanishes in a neighborhood of $x\in X'$, then $\K \mu (z)$ is smooth 
close to $x$.
\end{itemize}
\end{lemma}

\begin{proof}
Since $k(\zeta,z)$ is smooth in $z$ close to $x$ if $\zeta$ avoids a neighborhood of $x$, (ii) follows. 
To see (i) it is enough to assume that $\mu$ is smooth and has compact support close to $x\in X_{reg}$. 
Close to the point $(x,x)$ $X\times X$ is a smooth manifold, $\omega(\zeta)$ is smooth,
and $\hat{B}_{j} \sim |\zeta-z|^{-2j+1}$. Thus, (i)
follows from the following lemma, cf., the definition \eqref{k} and \eqref{kalt} of $k(\zeta,z)$.
\end{proof}

\begin{lemma}
Let $\Phi$ be a non-negative function on $\mathbb{R}^d_x\times \mathbb{R}^d_y$ such that 
$\Phi^2$ is smooth and $\Phi \sim |x-y|$. For each integer $m\geq 0$, let $\varphi_{m}$ denote an 
arbitrary smooth function that is $\ordo(|x-y|^m)$, and let $\mathcal{E}_{\nu}$ denote a
finite sum $\sum_{m\geq 0} \varphi_m/\Phi^{\nu+m}$. If $\nu\leq d-1$ and $\xi\in C^k_c(\mathbb{R}^d)$, then
\begin{equation*}
T\xi (x) = \int_{\mathbb{R}^d_y} \mathcal{E}_{\nu}(x,y) \xi(y) dy
\end{equation*}
is in $C^k(\mathbb{R}^d)$.
\end{lemma}

This lemma should be well-known, but for the reader's convenience we sketch a proof.
\begin{proof}[Sketch of proof]
Let $L_j=\partial/\partial x_j + \partial / \partial y_j$. It is readily checked (e.g., by Taylor expanding)
that $L_j \varphi_m = \varphi_m$ from which we conclude that $L_j \mathcal{E}_{\nu}=\mathcal{E}_{\nu}$. 
Let
\begin{equation*}
T^{\lambda}\xi (x)= \int_{\mathbb{R}^d_y} |x-y|^{2\lambda}\mathcal{E}_{\nu}(x,y) \xi(y) dy.
\end{equation*}
For $\real\, \lambda >-1/2$, it is clear that $T^{\lambda}\xi$ is an analytic $C^0(\mathbb{R}^d)$-valued function.
Moreover, for $\real\, \lambda >0$, one easily checks by using $L_j \mathcal{E}_{\nu}=\mathcal{E}_{\nu}$
that all distributional derivatives of order $\leq k$ of $T^{\lambda}\xi$ are continuous and 
analytic in $\lambda$ for $\real\, \lambda >-1/2$. It follows that $T\xi=T^0\xi\in C^k(\mathbb{R}^d)$. 
\end{proof}

\begin{proposition}\label{reglosning}
If $\mu$ is a section of $\mbox{Dom}_q\, \debar_X$ over $X$, then $\debar \mu \in \W_{0,q+1}(X)$ and 
the current equation
\begin{equation}\label{reglosn1}
\mu = \debar \K \mu + \K (\debar \mu) +\Proj \mu
\end{equation}
% \begin{equation*}
% \mu = \K (\debar \mu) + \Proj \mu, \quad q=0,
% \end{equation*}
holds on $X'_{reg}=X_{reg}\cap \Omega'$.
\end{proposition}

\begin{proof}
From Proposition~\ref{Bprop} it follows that $\debar \mu \in \W_{0,q+1}(X)$ and so $\K(\debar \mu)$
is a well-defined current in $\W(X')$. Moreover, from Lemma~\ref{ko} it follows that 
if $\phi(z)$ is a test form on $X'_{reg}$, then
\begin{equation}\label{tuba}
\phi(\zeta)=\int_{X_{z}} k(\zeta,z)\wedge \debar \phi(z)
+\debar_{\zeta} \int_{X_{z}} k(\zeta,z)\wedge \phi(z)
+\int_{X_{z}} p(\zeta,z)\wedge \phi(z)
\end{equation}
for $\zeta\in X_{reg}$.
We also see from Lemma~\ref{smooth} that all terms in \eqref{tuba} are smooth on $X$. 
If $\mu$ has compact support in $X_{reg}$, then the proposition follows by duality.

For the general case, let $\chi_{\delta}=\chi(|h|/\delta)$, where $h=h(\zeta)$ is a holomorphic 
tuple cutting out $X_{sing}$. Then the proposition holds for $\chi_{\delta}\mu$. 
Since $k(\zeta,z)\wedge \mu(\zeta)$ and $p(\zeta,z)\wedge \mu(\zeta)$ has the SEP on $X\times X'$,
we have that $\K(\chi_{\delta}\mu)\to \K \mu$ and $\Proj (\chi_{\delta}\mu)\to \Proj \mu$.
Moreover, $\debar \mu \in \W_{0,q+1}(X)$ so $k(\zeta,z)\wedge \debar \mu(\zeta)$ has the SEP, which 
implies that $\K(\chi_{\delta}\debar \mu)\to \K(\debar \mu)$. Hence,
\begin{equation*}
\lim_{\delta\to 0^+} \K(\debar(\chi_{\delta}\mu))=
\K(\debar \mu) + \lim_{\delta\to 0^+} \K(\debar\chi_{\delta}\wedge\mu).
\end{equation*}
The singularities of $k(\zeta,z)$ only  come from  the structure form
$\omega(\zeta)$ when $z$ and $\zeta$ ``far apart'', e.g.,
for $z$ in a compact subset of $X'_{reg}$ and $\zeta$ close to $X_{sing}$. From Lemma~\ref{urdjur}
we have that $\debar \chi_{\delta}\wedge\mu\wedge \omega\to 0$ and so 
$\lim_{\delta\to 0^+} \K(\debar\chi_{\delta}\wedge\mu)=0$ for $z$ in $X'_{reg}$; thus the proposition follows.
\end{proof}

Notice that $\Proj \mu$ in general is smooth. If the weight $g$ is holomorphic in $z$, 
then $\Proj \mu$ is holomorphic in $\Omega'$ for $q=0$ and $0$ for 
$q\geq 1$. In this case, Proposition~\ref{reglosning} thus is a homotopy formula for $\debar$ on $X'_{reg}$ 
in the sense that if $\mu$ is in $\mbox{Dom}_q\,\debar_X$ on $X$ and $\debar \mu=0$, then $\mu$ is holomorphic in $\Omega'$
for $q=0$ and $\mu=\debar \K\mu$ on $X'_{reg}$ for $q\geq 1$. 

\begin{proof}[Proof of Proposition~\ref{Wprop}]
We know that $\K\phi$ is defined and in $\W^X$ if $\phi\in \W^X$. By choosing
the weight $g$ to be holomorphic in $z$, we get that $\Proj \phi$ is in $\hol^X$.
Moreover,
from Proposition~\ref{reglosning} we have that the Koppelman formulas \eqref{Koppel1} and \eqref{Koppel2} hold
on $X'_{reg}$ if, in addition, $\phi$ is in $\mbox{Dom}\,\debar_X$. 

\end{proof}

We do not know whether $\debar \K\mu$ is in $\W^X$ or not, still less whether $\K \mu$ is in $\mbox{Dom}\,\debar_X$ or 
not in general. However, we shall now see that this is true 
if $\mu$ is smooth, and more generally if $\mu$ is obtained by a finite number of applications of $\K$'s.
Notice that $\K\mu$ is only defined in the slightly smaller set $X'$. Therefore, when we in the following lemma consider
products of kernels $\wedge_jk_j(z^j,z^{j+1})$, where 
$(z^1,\ldots,z^m)$ are coordinates on $X\times\dots \times X$, we will
assume that $z^{j+1}\mapsto k_{j+1}(z^{j+1},z^{j+2})$ 
has compact support where $z^{j+1}\mapsto k_j(z^j,z^{j+1})$ is defined.

\begin{lemma}[Main lemma]\label{mainlemma}
Let $k_j$ denote kernels \eqref{k} obtained via local embeddings and arbitrary Hermitian 
free resolutions of $\hol^{\Omega}/\J_X$. Let $(z^1,\ldots,z^m)$ be  coordinates on
$X\times \cdots \times X$ and assume that
$z^{j+1}\mapsto k_{j+1}(z^{j+1},z^{j+2})$ has compact support where $z^{j+1}\mapsto k_j(z^j,z^{j+1})$
is defined. Then, for any $x^m\in X$ and any $\omega\in \mathfrak{S}_{x^m}$, we have
\begin{equation}\label{snus}
\lim_{\delta\to 0^+}\debar \chi(|h(z^m)|/\delta)\wedge \omega(z^m)\wedge \bigwedge_{j=1}^{m-1} k_j(z^j,z^{j+1})=0
\end{equation} 
in the current sense in a neighborhood of $\{x^m\}\times X\times \cdots \times X$.
\end{lemma}

\begin{proof}
We proceed by induction over $m$. Every $\omega\in\mathfrak{S}_{x^m}$ is in $\W^X$, so $\chi_{\delta}\omega\to \omega$
and hence
\begin{equation*}
-\debar \chi_{\delta}\wedge \omega =\nabla_f(\chi_{\delta}\omega)\to
\nabla_f \omega =0.
\end{equation*} 
Thus the lemma holds for $m=1$ (i.e., when there are no $k$-kernels). 
Now consider the case $m+1$. Recall that the limit in \eqref{snus} is a pseudomeromorphic
current $T$ in a neighborhood of $\{x^{m+1}\}\times X\times \cdots \times X$. When $z^1\neq z^2$, then $k_1(z^1,z^2)$ 
is a smooth form times $\omega(z^1)$, cf., \eqref{kalt}.
Thus, outside $z^1=z^2$, $T$ is a smooth form times the tensor product of $\omega(z^1)$ and a current of the 
form \eqref{snus} in the variables $z^{m+1},\cdots,z^2$; the support of $T$ is thus contained in $\{z^1=z^2\}$
by the induction hypothesis. For a similar reason the support of $T$ must be contained in $\{z^{k}=z^{k+1}\}$ and 
we see that $T$ must have support contained in the diagonal $\Delta=\{z^{m+1}=\cdots =z^1=0\}$. Moreover, the support
of $T$ is clearly also contained in $X_{sing}\times X\times \cdots \times X$. Thus, the support of $T$ is contained in 
$(\Delta^X)_{sing}\subset \Delta$, which has dimension (at most) $n-1$ and hence codimension (at least)
$(m+1)n-(n-1)=mn+1$.

Now let $T^0$ be the component of $T$ obtained from the component $\omega_0(z^{m+1})$. 
Then $T^0$ has bidegree $(mn,m(n-1)+1)$ since each $k_j$ has bidegree $(n,n-1)$.
However, since $m\geq 1$, we have $m(n-1)+1< mn+1$ and so $T^0=0$ by the dimension principle. 
Let $T^1$ be the component of $T$ obtained from $\omega_1(z^{m+1})$. 
Since $\omega_1=\alpha^1\omega_0$ and $\alpha^1$ is smooth outside $X^1$, it follows from what we have just proved that 
$T^1$ has support contained in $(X^1\times X\times \cdots \times X)\cap \Delta \simeq X^1$. This set has 
codimension at least $mn+1+1$ and $T^1$ has bidegree $(*,m(n-1)+1+1)$ so also $T^1=0$ by the dimension principle.
Proceeding in this way we conclude that $T=0$.
\end{proof}

We can now show that Lemma \ref{ko} holds on $X\times X'$.

\begin{proposition}
We have that 
\begin{equation*}
-\nabla_{f(z)} \big(\omega(z)\wedge k(\zeta,z)\big) = 
\omega \wedge [\Delta^X] - \omega(z)\wedge p(\zeta,z)
\end{equation*}
in the current sense on $X\times X'$.
\end{proposition}

\begin{proof}
Let $\chi_{\delta}=\chi(|h(\zeta)|/\delta)$ and $\chi_{\epsilon}=\chi(|h(z)|/\epsilon)$, where $h$ as before
cuts out $X_{sing}$. From Lemma \ref{ko} we have that
\begin{equation*}
-\nabla_{f(z)} \big(\chi_{\delta}\chi_{\epsilon} \omega(z)\wedge k(\zeta,z)\big) =
\chi_{\delta}\chi_{\epsilon} \omega\wedge [\Delta^X] -
\chi_{\delta}\chi_{\epsilon}\omega(z)\wedge p(\zeta,z) + V(\delta,\epsilon),
\end{equation*}
where
\begin{equation*}
V(\delta,\epsilon)=\debar \chi_{\delta}\wedge \chi_{\epsilon} \omega(z)\wedge k(\zeta,z)+
\chi_{\delta}\debar\chi_{\epsilon}\wedge \omega(z)\wedge k(\zeta,z).
\end{equation*}
Since $\omega$, $k$, $p$, as well as the products $\omega(z)\wedge k(\zeta,z)$ and $\omega(z)\wedge p(\zeta,z)$ 
all are in $\W(X\times X)$, it is enough to see that $\lim_{\epsilon\to 0}\lim_{\delta\to 0}V(\delta,\epsilon)=0$.
We have
\begin{equation}\label{V}
\lim_{\delta\to 0} V(\delta,\epsilon)=
\lim_{\delta\to 0} \debar \chi_{\delta}\wedge \chi_{\epsilon} \omega(z)\wedge k(\zeta,z) +
\debar\chi_{\epsilon}\wedge \omega(z)\wedge k(\zeta,z).
\end{equation}
Since $\chi_{\epsilon}\omega(z)$ is smooth and vanishing in a neighborhood of $X_{sing}$,
$k(\zeta,z)$ is a smooth form times $\omega(\zeta)$, cf., \eqref{kalt}, on the support of $\debar \chi_{\delta}$
if $\delta$ is small enough. Therefore, the first term on the right hand side of \eqref{V} is $0$ by 
Lemma~\ref{mainlemma} with $m=1$. The second term on the right hand side of \eqref{V} tends to $0$ as 
$\epsilon\to 0$, again by Lemma~\ref{mainlemma}.
\end{proof}

\section{The ad~hoc sheaf $\A^X$}\label{asec}
%%and proofs of Theorems \ref{main} and \ref{koppelman} and Proposition~\ref{Wprop}}\label{Asektion}
We are now ready to define the sheaf $\A^X$; it is indeed an ad~hoc definition with respect
to the Koppelman formulas in the intrinsic context.
From the previous two sections we know that we locally 
(and semi-globally) on $X$ can construct integral kernels $k(\zeta,z)$ and $p(\zeta,z)$, cf., \eqref{kalt} and \eqref{palt},
and corresponding integral operators $\K$ and $\Proj$ such that Proposition~\ref{reglosning} holds.

\begin{definition}\label{Adefini}
We say that a $(0,q)$-current $\phi$ on an open set $\mathcal{U}\subset X$ is a section of $\A^X$ over $\mathcal{U}$, 
$\phi\in \A_q(\mathcal{U})$, if, for every $x\in \mathcal{U}$, the germ $\phi_x$ can be written as a finite sum of 
terms
\begin{equation*}
\xi_{\nu}\wedge \K_{\nu}(\cdots \xi_2\wedge \K_2(\xi_1\wedge \K_1(\xi_0))\cdots),
\end{equation*}
where $\K_j$ are integral operators with kernels $k_j(\zeta,z)$ at $x$ of the form defined in Section~\ref{KoppelsektionI}
and $\xi_j$ are smooth $(0,*)$-forms at $x$ such that $\xi_j$ has compact support in the set where 
$z\mapsto k_j(\zeta,z)$ is defined.
\end{definition}

Recall from Section \ref{KoppelsektionII} that if $\phi\in \W(\mathcal{U})$ and $\K$ is an integral 
operator, as defined above, with kernel $k(\zeta,z)$,
where $z\mapsto k(\zeta,z)$ is defined in $\mathcal{U}'\Subset \mathcal{U}$, then $\K \phi\in \W(\mathcal{U}')$. 
Therefore, $\A^X$ 
is a subsheaf of $\W^X$ and from Lemma \ref{smooth} it follows that the currents
in $\A^X$ are smooth on $X_{reg}$. In view of Lemmas~\ref{mainlemma} and \ref{urdjur}
we see that $\A^X$ is in fact a subsheaf of $\mbox{Dom}\, \debar_X$.
We also note that if $\phi\in \A_q(\mathcal{U})$, then $\K \phi \in \A_{q-1}(\mathcal{U}')$.

\begin{proof}[Proof of Theorem \ref{main}]
It is clear that $\A^X_q\supset \E^X_{0,q}$ are fine sheaves satisfying (i) of Theorem~\ref{main}
and we have just noted that also (ii) holds.

We must check condition (iii).
We have already seen in Proposition \ref{Bprop} that the kernel of $\debar$ in 
$\mbox{Dom}_0\,\debar_X$ is $\hol^X$.
Let $\phi$ be a section of $\A^X_q$, $q\geq 1$, in a neighborhood of an
arbitrary point $x\in X$, and assume that $\debar \phi=0$.
Since $\A^X\subset \mbox{Dom}\, \debar_X$ we also have $\debar_X\phi=0$. For some 
neighborhood $\mathcal{U}$ of $x$, by Proposition~\ref{reglosning}, we can find an operator $\K$ such that 
\begin{equation}\label{kaffe}
\debar \K \phi=\phi
\end{equation}
in $\mathcal{U}_{reg}$; here $\K$ corresponds to a weight that is holomorphic
in $z$. Since $\phi$ is a section of $\A^X_q$ we know that $\K \phi$ is a section of $\A_{q-1}^X$ and since
$\A^X \subset \mbox{Dom}\, \debar_X$ it follows from Proposition~\ref{Bprop} 
that $\debar \K \phi$ is in $\W^X$.
Both sides of \eqref{kaffe} thus have the SEP and we conclude that \eqref{kaffe} in fact holds on $\mathcal{U}$.

% Let $\chi_{\delta}=\chi(|h|/\delta)$, where $h$ is a holomorphic tuple cutting out 
% $X_{sing}$ and $\chi$ is a smooth approximation of $\mathbf{1}_{[1,\infty)}$. Since $\phi$ and 
% $\K\phi$ are in $\W$ we have $\chi_{\delta}\phi\to \phi$ and
% $\chi_{\delta}\K\phi\to \K\phi$, and moreover, in $\mathcal{U}'$, we have
% \begin{eqnarray*}
% \phi&=&\lim_{\delta\to 0^+} \chi_{\delta}\phi=
% \lim_{\delta\to 0^+}\chi_{\delta} \debar \K\phi=\lim_{\delta\to 0^+}\debar (\chi_{\delta}\K\phi)+
% \lim_{\delta\to 0^+} \debar \chi_{\delta}\wedge \K\phi \\
% &=& \debar \K\phi +0,
% \end{eqnarray*}
% where the last equality follows from Lemma~\ref{mainlemma}. Notice also that since $\K\phi$ is a section of 
% $\A^X\subset \mbox{Dom} \,\debar_X$, we have $\debar_X\K\phi=\debar\K\phi =\phi$ in $\mathcal{U}'$.

It remains to prove that $\debar$ is a map from $\A^X$ to $\A^X$. It is sufficient to show
that 
\begin{equation}\label{sova}
\debar\big(\xi_{\nu}\wedge \K_{\nu}(\cdots \xi_2\wedge \K_2(\xi_1\wedge \K_1(\xi_0))\cdots)\big)\in \A^X,
\end{equation}
for any operators $\K_j$ (not necessarily corresponding to weights that are holomorphic in $z$)
and smooth $(0,*)$-forms $\xi_j$ with compact support where $\K_j(\xi_{j-1})$ is defined. 
We prove \eqref{sova} by induction over $\nu$. The case $\nu=0$ is clear. Assume that \eqref{sova} holds for $\nu=\ell-1$.
Let $\K_j$, $j=1,\ldots,\ell$ be any integral operators and $\xi_j$, $j=0,\ldots,\ell$, 
smooth forms with compact support where $\K_j(\xi_{j-1})$ are defined.
Put $\phi_{\ell-1}=\xi_{\ell-1}\wedge \K_{\ell-1}(\cdots \xi_1\wedge \K_1(\xi_0)\cdots)$
and let $\mathcal{U}$ be a sufficiently small neighborhood of $\mbox{supp}\, \xi_{\ell}$. By Proposition~\ref{reglosning}
we have that
\begin{equation}\label{hej}
\phi_{\ell-1}=\K_{\ell} (\debar \phi_{\ell-1}) + \debar \K_{\ell} \phi_{\ell-1} + \Proj_{\ell} \phi_{\ell-1}
\end{equation}
in $\mathcal{U}_{reg}$; notice that $\Proj_{\ell}\phi_{\ell-1}$ is smooth.
From the induction hypothesis we have that $\debar \phi_{\ell-1}$ is in $\A^X$.
Moreover, any $\K$ maps $\A^X$ to $\A^X$ and since $\A^X\subset \mbox{Dom}\, \debar_X$,
all terms in \eqref{hej} have the SEP. Hence, \eqref{hej} holds on $\mathcal{U}$ and it follows
that $\debar \K_{\ell} \phi_{\ell-1}$ is in $\A(\mathcal{U})$.
Thus, \eqref{sova} holds for $\nu=\ell$ and the proof is complete.

%  Actually, \eqref{hej} holds on $\mathcal{U}$; to see this let $\chi_{\delta}$ be as in the 
% first part of the proof, multiply \eqref{hej} with $\chi_{\delta}$, and compute the limit using Lemma~\ref{mainlemma} as above.
% From the induction hypothesis we have $\debar \phi_{\ell-1}\in \A(X)$ and so it follows from \eqref{hej}
% that $\debar \K_{\ell} \phi_{\ell-1}\in \A(\mathcal{U})$. Thus, \eqref{sova} holds for $\nu=\ell$ and the proof is complete.
\end{proof}

\begin{proof}[Proof of Theorem~\ref{koppelman}]
From Section~\ref{KoppelsektionII} we have integral operators $\K$ and $\Proj$ such that 
$\Proj\varphi$ is holomorphic in $\Omega'$
if $\varphi\in \W_{0,0}(X)$ and $0$ if $\varphi\in \W_{0,q}(X)$, $q\geq 1$. Moreover, we 
noted above that $\K\colon \A_{q+1}(X)\to \A_q(X')$ and that $\A^X$
is a subsheaf of $\mbox{Dom}\, \debar_X$.  
Let $\phi\in \A_q(X)$, $q\geq 1$. By Proposition~\ref{reglosning} we have that 
\begin{equation}\label{gardin}
\phi = \debar \K \phi + \K(\debar \phi) + \Proj \phi
\end{equation}
on $X'_{reg}$. Since $\phi$ and $\debar \phi$ are in $\A^X$, all terms in \eqref{gardin}
have the SEP, cf., the previous proof. Hence \eqref{gardin} holds on $X'$
and so Theorem~\ref{koppelman} follows.
\end{proof}

\section{Example with a reduced complete intersection}\label{exsektion}

Let $a_1,\ldots,a_p\in \hol (\overline{\B})$, where $\B\subset \C^{n}$ is the unit ball, 
and assume that $X=\{a_1=\cdots =a_p=0\}\cap \B$ is a reduced complete intersection, i.e.,
that $X$ has pure codimension $p$ and $da_1\wedge \cdots \wedge da_p \neq 0$ on $X_{reg}$.
Let $e_1,\ldots,e_p$ be a holomorphic frame for the trivial
bundle $A$ and let $a$ be the section $a=a_1e_1^*+\cdots + a_pe_p^*$ of the dual bundle $A^*$,
where $\{e_j^*\}$ is the dual frame. Put $E_k=\Lambda^k A$ and let $\delta_a\colon \hol(E_{\bullet+1})\to \hol(E_{\bullet})$ be 
interior multiplication with $a$.
The Koszul complex $(\hol(E_{\bullet}), \delta_a)$ is 
then a free resolution of $\hol^{\Omega}/\J_X$, cf., \eqref{karvupplosn}.
It is clear that $s_a:=\sum_j \bar{a}_je_j/|a|^2$ is the solution to $\delta_a s_a=1$, outside $X$, with pointwise 
minimal norm (with respect to the trivial metric on $A$). If we consider all forms as sections of the bundle
$\Lambda (T^*(\Omega) \oplus A)$, then we can write \eqref{ukdef} as $u_k=s_a \wedge (\debar s_a)^{k-1}$.
Following \cite{AW1}, cf., \eqref{rdef}, we get that
\begin{equation}\label{CHprod}
R=R_p=\debar |a|^{2\lambda}\wedge u_p \big|_{\lambda=0}= \debar\frac{1}{a_p}\wedge \dots \wedge 
\debar \frac{1}{a_1} \wedge e_1 \wedge \cdots \wedge e_p,
\end{equation}
i.e., $R$ is the classical Coleff-Herrera product (times $e_1\wedge \cdots \wedge e_p$).
Let $\omega'$ be a smooth $E_p$-valued form in $\Omega\setminus X_{sing}$ such that 
$da_1\wedge \cdots \wedge da_p\wedge \omega' /(2\pi i)^p = e\wedge d\zeta$ where 
$e=e_1\wedge \cdots \wedge e_p$ and $d\zeta=d\zeta_1\wedge \cdots \wedge d\zeta_N$. 
Then the pullback $i^*\omega'$, where $i\colon X\hra \B$, is unique and meromorphic on $X$.
By the Leray residue formula we get that
\begin{equation*}
R\wedge d\zeta = \debar\frac{1}{a_p}\wedge \dots \wedge \debar \frac{1}{a_1}\wedge e \wedge d\zeta = 
\omega' \wedge [X],
\end{equation*}
and so, cf., \eqref{DLrep} and \eqref{staty}, the structure form associated to $R$ is 
$\omega:=i^*\omega'$.
If we choose coordinates $\zeta=(\zeta',\zeta'')$ so that $\det (\partial a/\partial \zeta')$ 
is generically non-vanishing on $X_{reg}$, then we can take 
$\omega'=(-2\pi i)^p \, e\wedge d\zeta'' /\det (\partial a/\partial \zeta')$ and the structure form is explicitly given as
\begin{equation*}
\omega = i^*\big( (-2\pi i)^{p} \, e\wedge d\zeta'' /\det (\partial a/\partial \zeta')\big).
\end{equation*}
If we let 
\begin{equation*}
\gamma = \frac{(-2\pi i)^p}{\det (\partial a/\partial \zeta')} \, e\wedge 
\frac{\partial}{\partial \zeta_p} \wedge \cdots \wedge \frac{\partial}{\partial \zeta_1},
\end{equation*}
then we have $R=(-1)^p\gamma \lrcorner [X]$, cf., \eqref{staty}.

Let $\mu\in \W_{0,q}(X)$ and assume that $\mu$ is smooth on $X_{reg}$. Then, cf., Section~\ref{starksektion}, 
$\mu$ is a section of $\mbox{Dom}\,\debar_X$
if and only if 
$\debar \chi_{\delta}\wedge \mu\wedge i^*(d\zeta''/\det (\partial a/\partial \zeta')) \to 0$ in the current sense as $\delta \to 0$;
here $\chi_{\delta}=\chi(|h|/\delta)$, $\chi$ is a smooth approximand of the characteristic function of 
$[1,\infty)$ and $h$ cuts out $X_{sing}$.

To construct integral kernels, cf., Section~\ref{KoppelsektionI}, let $h_j$ be $(1,0)$-forms 
so that $\delta_{\eta}h_j=a_j(\zeta)-a_j(z)$, where $\eta=\zeta-z$. 
We then have Hefer morphisms $H_k^{\ell}$ given as interior multiplication with 
$(\sum h_j\wedge e^*_j)^{k-\ell}/(k-\ell)!$. Let $g$ be the weight from Example~\ref{vikt}
and let $B$ be the Bochner-Martinelli form. Then
$(HR\wedge g\wedge B)_N =H_p^0R_p\wedge (g\wedge B)_n$ since $R=R_p$
and a straight forward computation shows that
\begin{equation*}
HR=H_p^0R_p=\debar \frac{1}{a_p}\wedge \cdots \wedge \debar \frac{1}{a_1}\wedge h_1\wedge \cdots
\wedge h_p.
\end{equation*}
There is a $\tilde{k}(\zeta,z)=\mathcal{O}(|\eta|^{-2n+1})$ such that
\begin{equation*}
h_1\wedge \cdots \wedge h_p\wedge (g\wedge B)_n=(2\pi i)^{-p}d\eta\wedge
\tilde{k}(\zeta,z)
\end{equation*}
and so from \eqref{k} we see that our solution kernel for $\debar$ on $X$ is
\begin{equation*}
k(\zeta,z)=\pm \big(\gamma \lrcorner H_p^0\wedge (g\wedge B)_n\big)_{(n)}
=\pm \tilde{k}(\zeta,z)\wedge \frac{d\zeta''}{\det(\partial a/\partial \zeta')}.
\end{equation*}
Similarly,
there is a smooth form $\tilde{p}(\zeta,z)$, depending holomorphically on $z$ if $g$ does, such that
\begin{equation*}
h_1\wedge \cdots \wedge h_p\wedge (\bar{\zeta}\cdot d\eta)\wedge (d\bar{\zeta}\cdot d\eta)^{n-1}=
(2\pi i)^{-p}d\eta \wedge \tilde{p}(\zeta,z)
\end{equation*}
and we can compute $p(\zeta,z)$ from \eqref{p} using $\tilde{p}(\zeta,z)$.
We get the representation formula
\begin{equation}\label{Stout}
\phi(z)=\int_{X}\debar \chi(\zeta) \wedge
\frac{d\zeta''}{\det(\partial a/\partial \zeta')}\wedge \frac{\tilde{p}(\zeta,z)}{(|\zeta|^2-z\cdot \bar{\zeta})^n}\,\phi(\zeta)
\end{equation}
for (strongly) holomorphic functions $\phi$ on $X$. If $X$ intersects $\partial \B$ properly and 
$X_{sing}$ avoids $\partial \B$ then we may let $\chi$ tend to the characteristic function for $\B$.
The integral \eqref{Stout} then becomes an integral over $X\cap \partial \B$ and the resulting representation
formula coincides with a formula of Stout \cite{Stout} and Hatziafratis \cite{Hatzi}.

\smallskip

Let us consider the cusp $X=\{a(z)=z_1^r-z_2^s=0\}\subset \B\subset \C^2$, 
where $2\leq r<s$ are relatively prime integers, in more detail.
In this case the structure form is the pullback of $-2\pi i \,e_1\wedge d\zeta_2/(r\zeta_1^{r-1})$ to $X$
and we can take $\gamma(\zeta)=(-2\pi i/r\zeta_1^{r-1})\cdot e_1\wedge (\partial/\partial \zeta_1)$.
The Hefer form is given by
\begin{equation*}
h=h_1d\eta_1 + h_2d\eta_2=
\frac{1}{2\pi i}\big(\frac{\zeta_1^r-z_1^r}{\zeta_1-z_1}d\eta_1 + \frac{\zeta_2^s-z_2^s}{\zeta_2-z_2}d\eta_2\big)
\end{equation*}
and we get
\begin{equation}\label{spets1}
h\wedge (g\wedge B)_1=h\wedge \chi(\zeta) \frac{\partial |\eta|^2}{2\pi i |\eta|^2}
=(2\pi i)^{-1} d\eta_1\wedge d\eta_2 \, \tilde{k}(\zeta,z)
\end{equation}
for a certain function $\tilde{k}(\zeta,z)$. The restriction of this function to $X\times X$ can be computed by applying
$\delta_{\eta}$ to \eqref{spets1} and noting that $\delta_{\eta}h=a(\zeta)-a(z)=0$ on $X\times X$. One gets that
$\tilde{k}(\zeta,z)=\chi(\zeta)\, h_1/\eta_2$ on $X\times X$ and so our solution kernel for $\debar$ on the cusp is
\begin{equation*}
k(\zeta,z)=\frac{d\zeta_2}{r\zeta_1^{r-1}}\, \tilde{k}(\zeta,z)=
\frac{\chi(\zeta)}{2\pi i} \frac{\zeta_1^r-z_1^r}{(\zeta_1-z_1)(\zeta_2-z_2)} \frac{d\zeta_2}{r\zeta_1^{r-1}}.
\end{equation*} 
Expressed in the parametrization $\tau\mapsto (\tau^s,\tau^r)=(\zeta_1,\zeta_2)$ our solution operator
for $(0,1)$-forms thus becomes
\begin{equation*}
\K \phi \,(t)= \frac{1}{2\pi i}\int_{\tau} \chi(\tau) \frac{\tau^{rs}-t^{rs}}{(\tau^{s}-t^s)(\tau^r-t^r)}
\frac{d\tau}{\tau^{(s-1)(r-1)}}\wedge \phi(\tau). 
\end{equation*}
One similarly shows that the projection operator $\Proj$ looks the same but with $\chi$ replaced by $\debar \chi$,
i.e., the kernel is the same but the solid integral is replaced by a boundary integral.

\section{The one-dimensional case}\label{kurvsektion}
In the case when $X$ is a complex curve we have some further results. In particular, we have
a stronger version of Proposition~\ref{Bprop}.

\begin{proposition}\label{kurvsats}
Let $X$ be a reduced complex curve.
\begin{itemize}
\item[(i)] If the complex 
$0\to \hol^X \hookrightarrow \E_{0,0}^X \stackrel{\debar}{\longrightarrow} \E_{0,1}^X \to 0$ is exact,
then $\A^X_*=\E^X_{0,*}$.

\item[(ii)] The complex 
$0\to \hol^X \hookrightarrow \mbox{Dom}_{0}\,\debar_X \stackrel{\debar}{\longrightarrow} \mbox{Dom}_{1}\,\debar_X \to 0$ 
is exact.
\end{itemize}
\end{proposition}

\begin{proof}
To prove (i), according to Definition~\ref{Adefini}, it is enough to show that $\K \xi$ 
is smooth for every $\K$
if $\xi$ is a smooth $(0,1)$-form. If $\xi$ is a smooth $(0,1)$-form,
there is (locally) a smooth function $\psi$ such that
$\debar \psi =\xi$. Smooth forms are in $\A^X$ and so, cf., the proof of Theorem~\ref{koppelman},
we get that
\begin{equation*}
\K \xi=\K(\debar \psi) = \psi - \Proj \psi
\end{equation*}
on $X$. Since $\Proj \psi$ is smooth, $\K \xi$ is indeed smooth on $X$. 

\smallskip

From Proposition \ref{Bprop} we have that the kernel of $\debar$ in $\mbox{Dom}_0\,\debar_X$ is $\hol^X$
so to prove (ii) it remains to see that $\debar \colon \mbox{Dom}_{0}\,\debar_X \to \mbox{Dom}_{1}\,\debar_X=\W_{0,1}^X$
is surjective. We take a minimal local embedding $X\hookrightarrow \C^N$ so that $X_{sing}=\{0\}$ and we let 
$\mu$ be a section of $\W_{0,1}^X$ in a neighborhood of $0$. We choose a Hermitian minimal free resolution of $\hol^X$ 
and we get the structure form $\omega=\omega_0$;
notice that $\hol^X$ is Cohen-Macaulay since $\mbox{dim}\,X=1$. Let $\K$ and $\Proj$ be integral operators
as in Section~\ref{KoppelsektionII} associated with a weight $g$ which is holomorphic in $z$.
From Proposition~\ref{reglosning} we have that $u_1:=\K\mu$ is in $\W_{0,0}^X$ and solves $\debar u_1=\mu$ outside $0$;
we will modify this solution to a solution in $\mbox{Dom}\, \debar_X$.

Let $\pi\colon \tilde{X}\to X$ be the normalization of $X$. Then $\tilde{\omega}:=\pi^*\omega$ 
is a meromorphic $(1,0)$-form and 
from \eqref{kommdiagram} we see that there is $\tilde{u}_1$ in $\W_{0,0}^{\tilde{X}}$ such that $\pi_* \tilde{u}_1=u_1$.
Let $h$ be a holomorphic tuple such that $\{h=0\}=\{0\}$ and put $\chi_{\delta}=\chi(|h|/\delta)$. 
Then $\nu:=\lim_{\delta\to 0^+}\debar\chi_{\delta}\wedge \tilde{u}_1 \tilde{\omega}$ is a 
pseudomeromorphic $(1,1)$-current on $\tilde{X}$ with support in the finite set of points $\pi^{-1}(0)$.  
Let us for simplicity assume that $X$ is irreducible at $0$ so that $\tilde{X}$ is connected and 
$\pi^{-1}(0)$ is just one point $t=0$ for some
holomorphic coordinate $t$ on $\tilde{X}$. 
Then $\nu$ has support at $t=0$ and hence equals a finite linear combination of 
derivatives of the Dirac mass, $\delta_0$, at $t=0$. Moreover, since $\nu$ is pseudomeromorphic, 
only holomorphic derivatives occur, cf., the first part of the proof of Proposition~\ref{grad/stod},
and so we have
\begin{equation*}
\nu= \sum_0^{\ell} c'_j\, \frac{\partial^j}{\partial t^j} \delta_0=
\sum_0^{\ell} c_j \, \debar \Big(\frac{1}{t^{j+1}}\Big)\w dt, \quad c'_j, c_j\in \C.
\end{equation*} 
Also, since $\tilde{\omega}$ is meromorphic, $\tilde{\omega}=f(t)dt/t^k$ for some $k\geq 0$ and some
holomorphic function $f$ with $f(0)\neq 0$. The current
\begin{equation*}
\tilde{u}_2:=\sum_{j=0}^{\ell} c_j\, \frac{t^{k-j-1}}{f(t)}.
\end{equation*}
is  holomorphic for $t\neq 0$ and by construction, $\nu=\debar (\tilde{u}_2\tilde{\omega})$. 
If $\tilde{u}:=\tilde{u}_1-\tilde{u}_2$, it is then straightforward to verify that 
$\debar\chi_{\delta}\wedge \tilde{u}\tilde{\omega}\to 0$ on $\tilde{X}$.
Hence, $u:=\pi_* \tilde{u} = u_1 - \pi_* \tilde{u}_2$ is in $\mbox{Dom}_0\,\debar_X$ and solves 
$\debar u=\mu$.
\end{proof}

Notice that once we know that $\debar \colon \mbox{Dom}_{0}\,\debar_X \to \mbox{Dom}_{1}\,\debar_X=\W_{0,1}^X$
is surjective, it is easy to show, using an argument similar to the proof of statement (i) above,
that our solution operators for $\debar$ indeed produce solutions
in $\mbox{Dom}_0\, \debar_X$. 

Also notice that, in view of Proposition~\ref{Bprop}, it follows from (ii) of Proposition~\ref{kurvsats}
that if $H^1(X,\hol^X)=0$ and $\phi\in \W_{0,1}(X) =\mbox{Dom}_1\,\debar_X$, 
then there is a $\psi\in \W_{0,0}(X)$ such that $\debar_X \psi=\phi$ on $X$.

\end{document}